\documentclass[10pt]{article}

\pdfoutput=1
\usepackage{hyperref}
\usepackage{amsmath,amsfonts,amssymb,latexsym,amsthm,dsfont}
\usepackage[a4paper,margin=2.5cm]{geometry}
\usepackage{graphicx}
\usepackage[utf8]{inputenc}
\usepackage[T1]{fontenc}
\usepackage{lmodern}
\usepackage{tikz,enumitem}
\usetikzlibrary{arrows,calc,decorations.text}
\usepackage[english]{babel}

\newtheorem{thm}{Theorem}[section]%
\newtheorem{cor}[thm]{Corollary}%
\newtheorem{prop}[thm]{Proposition}%
\newtheorem{lem}[thm]{Lemma}%

\newtheorem{defi}[thm]{Definition}%
\newtheorem{rem}[thm]{Remark}%
\newtheorem{asss}[thm]{Assertion}
\newcommand{\dE}{\mathbb{E}}

\newcommand{\dR}{\mathbb{R}}

\newcommand{\cA}{\mathcal{A}}

\newcommand{\cF}{\mathcal{F}}
\newcommand{\cH}{\mathcal{H}}

\newcommand{\cV}{\mathcal{V}}

\newcommand{\ind}{\mathds{1}}

\title{Comparison of the global dynamics for two chemostat-like models: random temporal variation versus spatial heterogeneity.}
\author{G. Lagasquie, S. Madec}

\begin{document}

\maketitle

\begin{abstract}
This article is dedicated to the study and comparison of two chemostat-like
competition models in a heterogeneous environment. The first model is a probabilistic model where we build a PDMP simulating
the effect of the temporal heterogeneity of an environment over the species in competition. Its study uses classical tools in this 
field. The second model is a gradostat-like model simulating the effect of the spatial heterogeneity of an environment over the same 
species. Despite the fact that the nature of the two models is very different, we will see that their long time behavior is globally very similar.
We define for both model quantities called invasion rates which model the growth rate of a species when
it is near to extinction. We show that the signs of these invasion rates essentially determine the long time behavior for both systems. 
In particular, we exhibit a new example of bistability between a coexistence steady state and a semi-trivial steady state. 
\end{abstract}


\section{Introduction}

The model of chemostat is a standard model for the evolution and the competition of several species for a single resource in an open environment. Its studies 
as well as that of its many variants has been widely explored since fifty years. One can read Smith and Waltman's book \cite{sw} and recent survey \cite{WADE201664} which 
give a view over the complexity and variability of this research domain. 
There are numerous applications for the chemostat.
For example, in population biology, the chemostat serves as a first
approach for the study of natural systems  . In industrial microbiology, the chemostat offers an economical production of micro-organisms.

Under various assumptions, the chemostat is known to satisfy the {\it principle of exclusive competition} which states that when 
several species compete for the same (single) resource, only one species survives, the one which makes ``best'' use of the resource (\cite{chem1,chem2,WX97,AM1980}). Though some 
natural observations and laboratory experiences support the principle of exclusive competition \cite{Hansen1491,grover2001}, the observed population diversity within 
some natural ecosystems seems to exclude it \cite{plancton1,plancton2}. In order to take account of the biological complexity without excluding the specificity of the 
chemostat, various models has been introduced (\cite{Loreau1,droop1973,Rapaport2008} for more examples).

The observed biodiversity could first be explained by the temporal fluctuations of the environment. This idea has been explored in the 
ecology literature (see for example \cite{eco1,eco2}). Applied to the chemostat, this idea gave \cite{detereco} where
the authors study the general gradostat with a periodic resource input. However, temporal fluctuations of an environment are most 
likely random. From this assumption comes the idea of studying an environment fluctuating randomly between a finite number of environments.
In \cite{lv}, the authors gives a complete study for a two-species Lotka-Volterra model of competition where the species evolve in an 
environment changing randomly between two environments and prove that coexistence is possible.


In order to take account of the biological 
complexity without excluding the specificity of the chemostat, Lovitt and Wimpenny introduced the gradostat model  which 
consists in the concatenation of various chemostats where the 
adjacent vessels are connected in both directions, \cite{grad1,grover2001}. The resource output occurs in the first and last chemostats of the chain and those 
in between exchange their contents.

The case where two species evolve in two interconnected chemostats is understood in various cases \cite{JSTW,SW2}.
See also \cite{STW,grad3,rapaport2010effects,grad5,grad4}  for more references on the general gradostat. 
The spatial heterogeneity has been also studies with partial differential equations models, see for instance \cite{sten1,sten2,HsuWalt1993} 

Some other chemostat-like model has been 
introduced to take account of the  temporal heterogeneity.  See
\cite{LENAS199461,Butler1985,detereco} with non autonomous deterministic model and in 
 \cite{CAMPILLO20112676,WANG2018341} with  stochastic models. 

In this article, we consider two species $u$ and $v$ competing for a single resource $R$. 
In a chemostat $\varepsilon$, we denote $\delta$ the common dilution rate for each species and the dilution rate of 
the resource, $R_0$ the constant input concentration of the resource in the vessel. For each species $w\in\{u,v\}$, let $f_w(R)$ be the consumption function. Thus, the {\it per capita} growth rate of the species $w$ is $f_w(R)-\delta$. Note that according to the models, $f_w$ can have different expressions. We 
choose here the most common expression for $f_w$ which is Monod's one:
\[
 f_w(R) = \frac{a_w R}{b_w +R}.
\]
where $a_w$ is the maximum growth rate for the species $w$ and $b_w$ is 'half-velocity constant' of the species $w$. 

Note $U(t)$, $V(t)$ and $R(t)$ the concentrations of the species $u$, $v$ and the resource $R$. 
The evolution of these different concentrations in the simple chemostat $\varepsilon$ is given by the equations:
\begin{equation}\label{eq:chem_simple}
  \left\{
\begin{aligned}
 &\dot{R}(t) = \delta (R_0 - R(t)) - U(t) f_u(R(t)) - V(t) f_v(R(t)) \\
 &\dot{U}(t) = U(t)\left(f_u(R(t))-\delta\right)\\
 &\dot{V}(t) = V(t)\left(f_v(R(t))-\delta\right)\\					
\end{aligned}
\right.
\end{equation}
together with the initial conditions $U(0)>0,\;V(0)>0\;R(0)\geq 0$.	
Let 
\[
 R_w^* = \begin{cases} 
	\frac{b_w \delta}{a_w - \delta} \text{ if $a_w>\delta$}\\
	+\infty \text{ if $a_w\leq\delta$},  
	\end{cases}
\]
be the concentration of resource satisfying $f_w(R_w) = \delta$ (if possible). This quantity $R_w^*$ can be interpreted as the minimal 
concentration of resource needed by the species $w$ to have its population growing. The species which needs the less resource to survive in the 
environment is the best competitor.

It is well known that the simple chemostat satisfies the principle of exclusive competition : only the best competitor 
survives. The following theorem\footnote{There exists various modification of the theorem \ref{thm:thmintro}. In 
particular, it is proven in \cite{AM1980} that the competitive exclusion principle holds true for general increasing consumption function $f_w$ verifying $f_w(0)=0$ and same dilution rates. It is yet unknown if the CEP holds true if the assumption on the dilution rates is relaxed. See \cite{WX97} for one of the last advance on this topic.} illustrates this 
statement
(see \cite{chem1,chem2}).

\begin{thm}[Competitive Exclusion Principle (CEP)]\label{thm:thmintro}
Suppose that $R_u^*<R_0$ ($u$ is able to survive) and $R_u^*<R_v^*$ ($u$ is the best competitor). The solutions of \eqref{eq:chem_simple} satisfy:
$$\underset{t\rightarrow +\infty}{\lim} (R(t),U(t),V(t)) = (R_u^*,R_0-R_u^*,0).$$
\end{thm}

\begin{rem}\label{rem:chgt_var}
Let us write:
\[
\Sigma(t) = R(t) + U(t) + V(t).
\]
Considering that the dilution rate is the same for every species and the substrat, it is easy to see that $\Sigma$ satisfies the 
differential equation:
\[
\dot{\Sigma}(t) = \delta(R_0 - \Sigma(t)).
\]
It comes that $\Sigma(t)=R_0+e^{-\delta t}(\Sigma(0)-R_0) \underset{t\rightarrow+\infty}{\longrightarrow} R_0$. 

 Using that $\Sigma(t)\to R_0$, it is classical (see the appendix F in \cite{sw}) that the asymptotic dynamics of the system \eqref{eq:chem_simple} is given by the dynamics of
the reducted system
 \begin{equation}\label{eq:chem_simple_reducted}
  \left\{
\begin{aligned}
 &\dot{U}(t) = U(t)\left(f_u(R_0-U(t)-V(t))-\delta\right)\\
 &\dot{V}(t) = V(t)\left(f_v(R_0-U(t)-V(t))-\delta\right)\\					
\end{aligned}
\right.
\end{equation}

Hence, assuming that the dilution rates are the same for every
species and the resource is a very strong hypothesis allows  to do the variable change $R(t) = R_0-U(t)-V(t)$. This is the key ingredient in \cite{AM1980} to prove the CEP for general increasing consumption functions and same dilution rates.
\end{rem}

In this paper, we consider two chemostats $\varepsilon^1$ and $\varepsilon^2$.
For $j\in\{1,2\}$, the parameters of the chemostat $\varepsilon^j$ are denoted $(R_0^{j},\delta^j,a_u^j,a_v^j,b_u^j,b_v^j)$. In all the article, the subscripts of a parameter or a variable make always reference to the species and the exponents make always reference to the environment. For a species $w\in\{u,v\}$, we set $\overline{w}\in\{u,v\}\setminus\{w\}$ the other species. With these two chemostats, we build two competition models.

The first model is a probabilistic one. In this model the chemostat where the two species and the resource evolve is alternating randomly
between $\varepsilon^1$ and $\varepsilon^2$. Assuming that the species and resource lives in $\varepsilon^1$ at $t=0$,
we wait a random exponential time of parameter $\lambda^1$ before switching the chemostat to $\varepsilon^2$. Then, we wait an other independent random exponential time of parameter $\lambda^2$ before switching back to $\varepsilon^1$, and so on.

The goal here is to model time
variations of the environment the species and resource evolve in. Mathematically, we build here a random process which study is totally different
from the gradostat model. In \cite{lv}, the authors study a similar process for a Lotka-Volterra competition model and we claim that it is possible
to adapt their techniques to the slightly more difficult chemostat switching competition model.

The second model is  a gradostat-like model where the two chemostats 
$\varepsilon^1$ and $\varepsilon^2$ are connected and trade their 
content at 
a certain rate $\lambda$. 
Mathematically, this model is a system of $3\times2$ differential equations which modelizes spatial heterogeneity in a 
biosystem (see \cite{grad1} for some mathematical results on the behavior of such system).

The goal of this article is to compare the long time behavior of the dynamics of these two different systems. {\em For each model }we give a 
mathematical definition for what we will call {\it the invasion rate of the species}, noted $\Lambda_w$ for the species $w$ in the 
probabilistic case\footnote{ In the deterministic case the invasion rate of the species $w$ is note $\Gamma_w$. However, we only refer to 
$\Lambda_w$ in this introduction.}. Given the mathematical difference between the two models, the definition of these invasion rates is different for each model.
However, we show that for each model, the sign of $\Lambda_u$ and $\Lambda_v$ essentially determines the 
state of the system at the equilibrium, and thus the long time dynamics. The precise results are state in the section \ref{section:temporal} for the probabilistic model 
and in the section \ref{section:spatial} for the deterministic model.

We show (under an additional assumption for the probabilistic case) that, if $\Lambda_u\Lambda_v> 0$, then for any positive initial condition only the two following behavior can happen for the two models. 

\begin{itemize}
 
 \item If $\Lambda_u<0$ and $\Lambda_v<0$ there is extinction of either species $u$ or species $v$  This configuration will be called
 the bi-stability.
 \item If $\Lambda_u>0$ and $\Lambda_v>0$ there is persistence of both species (persistence means that  $\liminf\limits_{t\to+\infty} U(t)>0$ and $\liminf\limits_{t\to+\infty} V(t)>0$).
\end{itemize}
In contrast, when $\Lambda_u \Lambda_v<0$, the possibilities for the long time dynamics are not exactly the same for the two models. For instance, if $\Lambda_u>0$ 
and $\Lambda_v<0$. Then in the probabilistic model their is always extinction of species $v$  but for the deterministic model there is 
either

 \begin{itemize} 		

\item Extinction of species $v$ (for almost all initial condition in $\Omega$).

\item Extinction of species $v$ or coexistence (depending on the initial condition in $\Omega$).

\end{itemize}

Consequently, comparing the two models will be essentially done by comparing the evolution of these 
invasion rates according to the parameter $\lambda$. An analytical and a numerical comparison of these invasion rates is done in 
section \ref{section:ccl}. In particular, we show, for the two models, that even if the two environments are favorable to the same 
species, then the two species may coexist or, worse, the other species is the only survivor.

For a more fluid reading , the technical proofs are postponed to section \ref{section:proofs}.
\section{Random temporal variation : model and main results.}\label{section:temporal}

\subsection{The probabilistic model :  a PDMP system}\label{subsection:temporalmodel}
As stated before, we pick two environments $\varepsilon^1$ and $\varepsilon^2$ and we model the environmental
variation of a biosystem by randomly switching the chemostat the two species and the resource evolve in.
This idea and its mathematical resolution has been introduced in \cite{lv}. In this previous article, the authors exhibit counterintuitive 
phenomenon on the behavior of a two-species Lotka-Volterra model of competition where 
the environment switches between two environments that are both favorable to the same species. Indeed, they show that 
coexistence of the two species or extinction of the species favored by the two environments can occur.

We consider the stochastic process $(R_t,U_{t},V_{t})$ defined by the system of differential equations:
\begin{equation}\label{eq:chem_pdmp}
  \left\{
\begin{aligned}
 &\dot{R_t} = \delta^{I_t} (R_0^{I_t} - R_t) - U_{t} f_u^{I_t}(R_t) - V_{t} f_v^{I_t}(R_t) \\
 &\dot{U_{t}} = U_{t}(f_u^{I_t}(R_t)-\delta^{I_t})\\
 &\dot{V_{t}} = V_{t}(f_v^{I_t}(R_t)-\delta^{I_t})
\end{aligned}
\right.
\end{equation}
where $(I_t)$ is a continuous time Markov chain on the space of states $E=\{1,2\}$. We note $\lambda^1$ and $\lambda^2$ the jump rates.
Starting from the state $j$, we wait an exponential time of parameter $\lambda^j$ before jumping to the state $\overline{j}$. The invariant measure
of $(I_t)$ is $\frac{\lambda^2}{\lambda^1+\lambda^2}\Delta^1 + \frac{\lambda^1}{\lambda^1+\lambda^2}\Delta^2$ (where $\Delta^j$ 
is the Dirac measure in $j$). 

Let us note the jump rates: $\lambda^1 = s\lambda$ and $\lambda^2 = (1-s)\lambda$ with $s \in (0,1)$ and $\lambda>0$. Parameter $s$ (respectively $1-s$) can be seen
as the proportion of time the jump process $(I_t)$ spends in state $2$ (respectively $1$). The parameter $\lambda$ will be seen as the global switch rate of $(I_t)$.

The process $(Z_t) = (R_t,U_{t},V_{t},I_t)$ is what we call a Piecewise Deterministic Markov Process (PDMP) as introduced
by Davis in \cite{davis}. 

Let us call:
\[
K = \{(r,u,v)\in \dR_+^{3} \text{ }, \quad \frac{\min(R_0^1,R_0^2)}{2} \leq r+u+v \leq 2\max(R_0^1,R_0^2) \},
\]
and 
\[
M = K\times\{1,2\}.
\]
According to remark \ref{rem:chgt_var}, $Z_t$ will reach $M$ for any initial condition $Z_0 \in \dR_+^{n+1}\times\{1,2\}$. We can then assume 
that $Z_0\in M$ and, as a consequence, $M$ is as the state space of the process $(Z_t)$.

We will call the extinction set of species $w$ the set:
\[
 M_{0,w} = \{ (r,u,v,i) \in M\text{ },\quad w = 0\},
\]
and the extinction set:
\[
 M_0 = \underset{w}{\bigcup} M_{0,w}.
\]
It is clear that the process $(Z_t)$ leaves invariant all the extinction set and the interior set $M\setminus M_0$.

In order to describe the behavior of the process $(Z_t)$ when $Z_0 \in M\setminus M_0$, \cite{lv} suggests to study the invasion rates 
of species $w$ defined as:
\[
 \Lambda_w = \int \left(f_w^{1}(R)-\delta^{1}\right) d\mu(R,1) + \int \left(f_w^{2}(R)-\delta^{2}\right) d \mu(R,2),
\]
where $\mu$ is an invariant probability measure of $(Z_t)$ on $M_{0,w}$.

\begin{rem}
The idea behind the definition of the invasion rate $\Lambda_u$ (same for $\Lambda_v$) is the following. From (\ref{eq:chem_pdmp}) comes:
\begin{align*}
& \frac{\dot{U_t}}{U_t} = f_u^{I_t}(R_t)-\delta^{I_t} = \cA(Z_t)\\
& \int \frac{\dot{U_t}}{U_t} ds = \int \cA(Z_s)ds\\
& \frac{1}{t} \log U_t = \frac{1}{t} \int \cA(Z_s)ds.
\end{align*}
Formally, the ergodic theorem allows to write:
\[
 \frac{1}{t} \log U_t \rightarrow \int \cA(z) d\mu(z),
\]
where $\mu$ is an invariant measure for the process $(Z_t)$. If $\mu$ is an invariant measure of $(Z_t)$ on $M_{0,u}$, we define 
$\Lambda_u = \int \cA(z) d\mu(z)$. By Feller continuity (see \cite{mal1}) it comes that $\Lambda_u$ can be seen as the exponential 
growth rate of $U_t$ when $U_t$ is close to zero.
\end{rem}

As stated in this previous remark, $\Lambda_w$ can be seen as the exponential growth rate of the concentration of the species $w$ when its concentration is close to zero. If 
$\Lambda_w > 0$, the concentration of $w$ tends to increase from low values and if $\Lambda_w<0$, the concentration of $w$ tends to decrease from low values.

\subsection{Dynamics of the PDMP model}\label{subsection:temporalresults}

We are interested in the long time behavior of the concentration of the species $u$ and $v$. In \cite{lv}, the authors show that the signs of the invasion rates
characterizes the long time behavior of the randomly switched Lotka-Volterra model of competition. It is expected to have the same result in the chemostat
case. We expect the two following behavior for the concentration of the species $u$ and $v$:
\begin{itemize}
 \item Species $w\in\{u,v\}$ goes to extinction if $W_t \rightarrow 0$ almost surely for any initial condition $Z_0 \in M\backslash M_0$.
 \item We have coexistence of the two species when the two species do not go to extinction for any initial condition $Z_0 \in M\backslash M_0$.
In this case, any invariant measure of $(Z_t)$ is supported by $M\backslash M_0$.
\end{itemize}

We will not give the proofs for theorem \ref{thm:intro_n1} and theorem \ref{thm:intro_n2} since it follows to a few details the same 
path as in \cite{lv}. Note that these proofs uses some renewal theory arguments coupled with the analytic properties of the invasion rates.

\subsubsection{Long time behavior when only one species is introduced}
Assume that species $\overline{w}$ is not in the system ($\overline{W_t} = 0$). Then, the process $Z_t = (R_t,W_t,I_t)$ satisfies:
\begin{equation}
  \left\{
\begin{aligned}
 &\dot{R_t} = \delta^{I_t} (R_0^{I_t} - R_t) -  U_{t} f_u^{I_t}(R_t)\\
 &\dot{W_{t}} = W_{t}(f_w^{I_t}(R_t)-\delta^{I_t})
\end{aligned}
\right.
\end{equation}
In order to emphasize the fact that species $\overline{w}$ is absent of the system, let us define:

\[
 \Lambda_w^0 = \int \left(f_w^{1}(R)-\delta^{1}\right) d\mu_w^0(R,1) + \int \left(f_w^{2}(R)-\delta^{2}\right) d \mu_w^0(R,2),
\]
where $\mu_w^0$ will be proven (see Section 4 and 5) to be the unique invariant measure of the process ($Z_t)$ restricted to $M_{0,w}$.



The first result which is similar to the main result in \cite{lv} is the following:

\begin{thm}\label{thm:intro_n1}
The sign of the invasion rate $\Lambda_w^0$ characterizes the evolution of the species $w$:
   \begin{enumerate}
     \item If $\Lambda_w^0<0$ species $w$ goes to extinction: $W_t \rightarrow 0$ almost surely.
     \item If $\Lambda_w^0>0$ species $w$ perpetuates.
   \end{enumerate}
\end{thm}

\subsubsection{Long time behavior when two species are introduced}

Now, we assume that \underline{$R_0^1=R_0^2=R_0$}. According to remark \ref{rem:chgt_var},
the sum $\Sigma_t\to R_0$ as $t\to+\infty$. As a consequence, the long-time behavior of $(Z_t)$ is obtained by
assuming that $\Sigma_t = R_0$ in \eqref{eq:chem_pdmp}. Moreover, recall that the invasion rates are defined by:
\[
 \Lambda_w = \int \left(f_w^{1}(R)-\delta^{1}\right) d\mu_w(R,1) + \int \left(f_w^{2}(R)-\delta^{2}\right) d \mu_w(R,2),
\]
where $\mu_w$ is an invariant measure of $(Z_t)$ restricted to $M_{0,w}$.

\begin{asss}
Denote $(H_w)$ the assertion which is true if and only if one of the following assertion is true :
\begin{itemize}
 \item (i) $\exists j \in \{1,2\}$ such that $\varepsilon^j$ is unfavorable to the species $w$.
 \item (ii) $\exists s\in(0,1)$ such that the averaged chemostat $\varepsilon_s$ is unfavorable to the species $w$ (see the following remark \ref{rem:chem_moyen} for a precise definition of the averaged chemostat).
\end{itemize}
\end{asss}

\begin{rem}\label{rem:chem_moyen}
Formally, let  $\varepsilon_s=(1-s)\varepsilon^1 + s \varepsilon^2$ the averaging of the two chemostats $\varepsilon^1$ and $\varepsilon^2$. 
The associated differential system modelizing the behavior of the different
concentrations in $\varepsilon_s$ is given by:
\begin{equation}
 \left\{
\begin{aligned}
 &\dot{R} = \overline{\delta} (\overline{R_0} - R) -U \overline{f_u}(R)-V \overline{f_v}(R)\\
 &\dot{U} = U(\overline{f_w}(R)-\overline{\delta})\\
&\dot{V} = V(\overline{f_w}(R)-\overline{\delta})\\
\end{aligned}
\right.
\end{equation}
Where $\overline{\delta} = (1-s)\delta^1 + s\delta^2$, $\overline{f_w} = (1-s)f_w^1 + sf_w^2$ and:
\[
 \overline{R_0} = \frac{(1-s)\delta^1R_0^1+s\delta^2R_0^2}{\overline{\delta}}.
\]

Despite the fact that the averaged consumption functions  $\overline{f_w}$ are not Monod functions in general, they are increasing functions verifying $\overline{f_w}(0)=0$. Thus the PEC holds for $\varepsilon_s$. In this sense we can defined the best competitor in $\varepsilon_s$. The averaged chemostat $\varepsilon_s$ is saied to be unfavorable to a species $w\in\{u,v\}$ if $w$ is not the best competitor in $\varepsilon_s$, that is if $W(t)\to 0$ as $t\to +\infty$. 
\end{rem}

Once again, the signs of the invasion rates  $\Lambda_u$, $\Lambda_v$ fully describe the long time behavior of the process:

\begin{thm}\label{thm:intro_n2}
The sign of the invasion rates $\Lambda_u$, $\Lambda_v$ characterizes the evolution of the species:
  \begin{enumerate}
     \item If $\Lambda_u>0$ and $\Lambda_v<0$ and $(H_v)$ is true then species $v$ goes to extinction.
     \item If $\Lambda_u<0$ and $(H_u)$ is true and $\Lambda_v>0$ then species $u$ goes to extinction.
     \item If $\Lambda_u<0$ and $\Lambda_v<0$ then of one the species goes to extinction. We say that it is a situation of exclusive bistability..
     \item If $\Lambda_u>0$ and $\Lambda_v>0$ then there is coexistence of both species.
  \end{enumerate}
\end{thm}


See section 4 for a numerical investigation over the signs of these invasion rates. We show numerically that for any couple of signs 
$(x,y) \in\{+,-\}$ there exists pair of chemostats $\varepsilon^1$, $\varepsilon^2$  such that $(Sign(\Lambda_u),Sign(\Lambda_v)) = (x,y)$.

Moreover, $\varepsilon^1$ and $\varepsilon^2$ may be chosen both favorable to $u$ ($R_u^j<R_v^j$ for $j=1,2$) or both favorable to $v$ ($R_u^j>R_v^j$ for $j=1,2$) or one favorable to $u$ and the other to $v$ ( ($R_u^1-R_v^1)(R_u^2-R_v^2)<0$ for $j=1,2$).

In particular, it is possible to pick chemostats $\varepsilon^1$ and $\varepsilon^2$ both favorable to the species $u$ such that for some values of the switching rate $\lambda$, 
$\Lambda_u <0$: switching between two environments favorable to species $u$ can surprisingly make it disappear (see figure \ref{fig:fig2}-a).

\section{Spatial heterogeneity : model and main results}\label{section:spatial}
\subsection{The deterministic model : a gradostat-like system}\label{subsection:spatialmodel}
The gradostat model is obtained by connecting the two chemostats $\varepsilon^1$ and $\varepsilon^2$ and allowing them to trade their content.

Note $\cV^j$ the volume of the chemostat $\varepsilon^j$ and $Q$ the volumetric flow rate between the two vessels and $U^j(t)$ the concentration
of the species $u$ in the chemostat $\varepsilon^j$. It comes:
\[
 \left\{
\begin{aligned}
 &\dot{(U^1\cV^1)}(t) = -QU^1(t) + QU^2(t)\\
 &\dot{(U^2\cV^2)}(t) = QU^1(t) - QU^2(t).
\end{aligned}
 \right.
\]
Which implies the following differential equations on the concentrations:
\begin{equation}\label{eq:transfert}
 \left\{
\begin{aligned}
 &\dot{U^1}(t) = -\frac{Q}{\cV^1}U^1(t) + \frac{Q}{\cV^1}U^2(t)\\
 &\dot{U^2}(t) = \frac{Q}{\cV^2}U^1(t) - \frac{Q}{\cV^2}U^2(t).
\end{aligned}
 \right.
\end{equation}
We will denote $\lambda^j = \frac{Q}{\cV^j}$. Similarly, we denote $V^j(t)$ the concentration of the species $v$ in the chemostat 
$j$ and $R^j(t)$ the concentration of the resource in the chemostat $j$. We will also denote $\{j,\overline{j}\} = \{1,2\}$. 

The evolution of the gradostat is described by the following 
system of  differential equations:
\begin{equation}\label{eq:chem_Deter}
 \left\{
\begin{aligned}
 &\dot{R^j}(t) = \delta^j (R_0^{j} - R^j(t)) - U^j(t) f_u^j(R^j(t)) - V^j(t) f_v^j(R^j(t)) \;\;+{\lambda^j (R^{\overline{j}}(t)-R^j(t))}\\
 &\dot{U^j}(t) = U^j(t)(f_u^j(R^j(t))-\delta^j)\phantom{++++++++++++++}+{\lambda^j (U^{\overline{j}}(t)-U^j(t))}\\
 &\dot{V^j}(t) = V^j(t)(f_v^j(R^j(t))-\delta^j)\phantom{++++++++++++++}+{\lambda^j (V^{\overline{j}}(t)-V^j(t))}.
\end{aligned}
\right.
\end{equation}
The part with $\lambda^j$ in factor comes from the transfer equation \eqref{eq:transfert} and the other part comes from the chemostat equation
\eqref{eq:chem_simple}.

Let us write $R(t) = \begin{pmatrix}
 R^1(t)\\
 R^2(t)
\end{pmatrix}$, $U(t) = \begin{pmatrix}
 U^1(t)\\
 U^2(t)
\end{pmatrix}$, $V(t) = \begin{pmatrix}
 V^1(t)\\
 V^2(t)
\end{pmatrix}$, $R_0 = \begin{pmatrix}
 R_0^{1}\\
 R_0^{2}
\end{pmatrix}$, $\delta = \begin{pmatrix}
 \delta^1\\
 \delta^2
\end{pmatrix}$ and $f_w(R) = \begin{pmatrix}
 f_w^1(R^1)\\
 f_w^2(R^2)
\end{pmatrix}$. Moreover, set $\lambda^1 = s\lambda$ and $\lambda^2 = (1-s)\lambda$ with $\lambda >0$ and $s\in(0,1)$ and $K = \begin{pmatrix}
 -s & s\\
1-s & s-1
\end{pmatrix}$. By convention $\begin{pmatrix}
 w\\
 x
\end{pmatrix} \begin{pmatrix}
 y\\
 z
\end{pmatrix} = \begin{pmatrix}
 wy\\
 xz
\end{pmatrix}$. With this notations, the system \eqref{eq:chem_Deter} reads shortly:
\begin{equation}\label{eq:chem_deter_r}
 \left\{
\begin{aligned}
 &\dot{R}(t) = \delta (R_0 - R(t)) - U(t) f_u(R(t))+\lambda K R(t)\\
 &\dot{U}(t) = U(t)(f_u(R(t))-\delta)\qquad\qquad\;+\lambda K U(t)\\
 &\dot{V}(t) = V(t)(f_v(R(t))-\delta)\qquad\qquad\;+\lambda K V(t).
\end{aligned}
\right.
\end{equation}
We consider initial value will be taken in the set $(\dR_+^*\times\dR_+^*)^3$.

Set $\Sigma^j(t) = R^j(t) + U^j(t) + V^j(t)$. The vector $\Sigma(t) = \begin{pmatrix}
                                                                        \Sigma^1(t)\\
                                                                        \Sigma^2(t)
                                                                       \end{pmatrix}$ 
satisfies the linear differential system:
\[
  \dot{\Sigma}(t) = \left( \lambda K - \Delta \right) \Sigma(t) + \delta R_0,
\]
where $\Delta = \begin{pmatrix}
                 \delta^1 & 0\\
                 0 & \delta^2\\
                \end{pmatrix}$.

The matrix $\lambda K - \Delta$ has two real negative eigenvalues.
Hence we may set $\Sigma=\begin{pmatrix}\Sigma^1\\\Sigma^2\end{pmatrix}:=(\Delta-\lambda K)^{-1}(\delta R_0)$ and we have

$$\lim_{t\to+\infty}\Sigma(t)=\Sigma$$


Since every trajectory is asymptotic to its omega limit set, it is important to study the system on this set. 

As a consequence, in all the following
our attention will be focused on the system:
\begin{equation}\label{eq:chem_deter}
 \left\{
\begin{aligned}
 &\dot{U}(t) = U(t)(f_u(\Sigma-U(t)-V(t))-\delta)+\lambda K U(t)\\
 &\dot{V}(t) = V(t)(f_v(\Sigma-U(t)-V(t))-\delta)+\lambda K V(t).
\end{aligned}
\right.
\end{equation}
With initial condition in the set $(\dR_+^{*}\times \dR_+^{*})^2$.
The appendix $F$ of \cite{sw} shows that the long time dynamics of \eqref{eq:chem_Deter} is completly given by the dynamics of \eqref{eq:chem_deter}.

\subsection{Dynamics of the gradostat like model}\label{subsection:spatialresults}

We are interested in the long time behavior of the solution of this differential system. It is proven in \cite{sw,JSTW}, using strongly the monotonicity of the system, that
any solution of \eqref{eq:chem_deter} converges to a stationary equilibrium when the consumption functions $f_w^j$ do not depend on the vessel
$\varepsilon^j$. Their proofs are mainly based on the study of the existence and stability of stationary solutions and on general results about monotone system due to Hirsch (see the appendix B and C in \cite{sw} and the references therein). 

This strategy is still working in the case of vessel-dependent consumption function $f_w^j$, the main additional difficulty being that the 
structure of the stationary solutions is richer when the functions $f_w^j$ do depend on $j$.  We do a 
complete description of the stationary solution detailled in section \ref{section:proofs}. This description relies on the construction of different
functions defined on the interval $[0,R_0^1]$ which intersections in a certain domain of the plane $[0,R_0^1]\times[0,R_0^2]$ gives the existence and stability
of stationary solutions for \eqref{eq:chem_deter}.
 
The main idea of the construction of these functions is the following:
\begin{enumerate}
\item If the species $w$ survives at the equilibrium, then $0$ is the principal eigenvalue of the matrix $A_w(R)=f_w(R)-\delta+\lambda K$ which implies that $R=(R^1,R^2)$ belongs to the graph of a function $F_w$.
\item If the  species $w$  survives (without competition) then $W=R_0-R$ is the principal eigenfunction of $A_w(R)$ and then $R=(R^1,R^2)$ belongs to the graph of a function $g_w$.
\end{enumerate}

In section \ref{section:proofs}, it is show how the relative position of the four curves  $R^2=g_w(R^1)$ and $R^2=F_w(R^1)$ ($w\in\{u,v\}$) give a graphical understanding of the existence of the steady states and their stability. See the  figure 
\ref{fig:caspossibles}.


\subsubsection{Long time behavior when only one species is introduced}
Assume that $\overline{w}$ is not in the system ($\overline{W}(t) = 0$). In this particular case, it is possible to study 
the behavior of the system. Without competition, the differential equation describing the evolution of the system is:
\begin{equation}\label{eq:chem_deter_1}
\dot{W}(t) = W(t)(f_w(\Sigma-W(t))-\delta)+\lambda K W(t)
\end{equation}
with initial condition$W(0)\in\dR_+^*\times\dR_+^*$.

It can be proven like in \cite{sw} that any trajectory of this previous differential equation goes to a stationary point. Let us call $E_0 =(0,0)$, 
$E_0$ is the trivial stationary point of the system (\ref{eq:chem_deter_1}) and its linear stability gives the characterizes of the solutions of
\eqref{eq:chem_deter_1}:

\begin{thm}\label{thm:semitrivial} The global dynamics of the system \eqref{eq:chem_deter_1} is as follows.
\begin{itemize}
 \item If $E_0$ is linearly stable, then it is the only stationary point and any 
trajectory is attracted by $E_0$ for any initial condition in $\dR_+^{*,2}$.
 \item If $E_0$ is linearly unstable, then there exists a unique stationary point $E_w = (W^1,W^2)\in\dR_+^*\times\dR_+^*$. Moreover
$E_w$ is a global attractor for the system \eqref{eq:chem_deter_1} in $\dR_+^*\times\dR_+^*$.
\end{itemize}
\end{thm}

Note that a stationary point for equation \eqref{eq:chem_deter_1} satisfies the equation:
\[
 \cF_w(W) = W(f_w(\Sigma-W)-\delta)+\lambda K W =0.
\]
The jacobian matrix of $\cF_w$ taken at $E_0$ is:
\begin{equation}\label{eq:jaco}
 A_w = \begin{pmatrix}
      f_w^1(\Sigma^1)-\delta^1-\lambda^1 & \lambda^1\\
      \lambda^2 & f_w^2(\Sigma^2)-\delta^2-\lambda^2
     \end{pmatrix}.
\end{equation}
We define the invasion rate $\Gamma_w^0$ of the species as the maximum eigenvalue of the matrix $A_w$:
\begin{equation}\label{eq:gaga}
 \Gamma_w^0 = \frac{1}{2} \left( f_w^1(\Sigma^1)-\delta^1 + f_w^2(\Sigma^2)-\delta^2 -\lambda^1 - \lambda^2 + \sqrt{\left( f_w^1(\Sigma^1)-\delta^1 - f_w^2(\Sigma^2)+\delta^2 \right)^2 + 4\lambda^1\lambda^2} \right)
\end{equation}

Theorem \ref{thm:semitrivial} yields:

\begin{cor}\label{cor:gammau}
The sign of $\Gamma_w^0$ characterizes the behavior of the system \eqref{eq:chem_deter_1}:
\begin{itemize}
 \item If $\Gamma_w^0 <0$ there is extinction of the species $w$: $\lim\limits_{t\to +\infty}W(t)=0$.
 \item If $\Gamma_w^0 >0$  there is persistence of the species $w$. More precisly:  $\lim\limits_{t\to +\infty}W(t)=E_w\in\dR_+^*\times\dR_+^*$.
\end{itemize}

\end{cor}


%
%

\subsubsection{Long time behavior when two species are introduced}%
For sake of comparison with the probabilistic case, we set 
$R_0 = R_0^1 = R_0^2$ even if computations are possible when these two quantities are different. The system  \eqref{eq:chem_deter} being strongly monotone,  the theorem C.9 from Hirsch  \cite{sw} implies that for almost all initial condition, the solutions tends to a stationary point.
Thus, the study of the existence and stability of these solutions is crucial in the understanding of the 
long-time behavior of the solutions of

From $R_0^1 = R_0^2$,  a stationary solution of \eqref{eq:chem_deter} 
satisfyies:
\begin{equation}\label{eq:station1}
 \cH(U,V) = 0 \Leftrightarrow \left\{
\begin{aligned}
 &U(f_u(R_0-U-V)-\delta)+\lambda K U = 0\\
 &V(f_v(R_0-U-V)-\delta)+\lambda K V = 0.
\end{aligned}
\right.
\end{equation}

Set $E_0 = (0,0,0,0)$. $E_0$ is the trivial stationary equilibrium. The jacobian matrix of $\cH$ at $E_0$ reads:
\[
 dH(E_0) = \begin{pmatrix}
  A_u & 0\\
  0 & A_v
 \end{pmatrix}
\]
where $A_w$ is defined in \eqref{eq:jaco}.

If both $A_u$ and $A_v$ have negative eigenvalues then $E_0$ is a locally attractive stationary point, and there are no 
other stationary equilibrium points.

If $A_u$ has at least one positive eigenvalue, then $E_0$ is not locally attractive. As a consequence, theorem \ref{thm:semitrivial} from
the previous subsection gives the existence of a unique semi-trivial stationary equilibrium $E_u=(U,0)$. Likewise, if $A_v$ has at least
one positive eigenvalue, we define $E_v = (0,V)$ as the other semi-trivial stationary equilibrium.

Moreover, arguments similar to the ones in  \cite{sw} yield
\begin{prop}\label{prop:global}
\begin{itemize}
\item  If $E_u$ and $E_v$ does not exists, then $E_0$ is a global attractor.
\item  Let $\{w,\overline{w}\}=\{u,v\}$.  If $E_w$ exists and $E_{\overline{w}}$ does not exists, then $E_w$ is a global attractor.
\end{itemize}
\end{prop}

Hence, the most interesting case holds when both $E_u$ and $E_v$ exists.
In that case, it is possible to have  coexistence stationary solutions which may be stable or unstable. 



 

Define the following matrix:
\begin{equation}\label{Mw}
 M_{\overline{w}}(R_w) = \begin{pmatrix}
        f_{\overline{w}}^1(R_w^1) -\delta^1-\lambda^1 & \lambda^1 \\
        \lambda^2 & f_{\overline{w}}^2(R_w^2) -\delta^2-\lambda^2
       \end{pmatrix}.
\end{equation}
We show in section \ref{subsection:graphical} that  the stability of the semi-trivial equilibrium $E_w$  is given by the sign of the eigenvalues of $M_{\overline{w}}(R_w)$.

%

\begin{defi}\label{defi:gamma}
Let $\Gamma_w$ be the maximum eigenvalue of the matrix $M_w(R_{\overline{w}})$. We call $\Gamma_w$ the invasion rates
of the species $w$. 
\end{defi}
\begin{rem}
Let us explain the designation ``invasion rate'' for $\Gamma_u$.
If $\Gamma_u>0$, it means that the semi-trivial equilibrium $E_u=(U^1,U^2,0,0)$ is unstable. Consequently, according to previous remark,
it means that $(0,0)$ is un unstable equilibrium for the differential system:
\[
 \dot{V}(t) = V(t)\left(f_v(R_0-U-V(t))-\delta \right) +\lambda K V(t).
\]
Hence, if $V(0)$ is small enough, then $t\mapsto V(t)$  is increasing on $(0,\tau)$. In other words, $v$ invade the environment.
At the contrary, if $\Gamma_u<0$, the semi-trivial equilibrium $E_u$ is stable and from a small initial value $V(0)$, 
$V(t)\to (0,0)$.
\end{rem}

The signs of the invasion rates $\Gamma_w$ give the stability of the semi-trivial equilibrium $E_{\overline{w}}$ but determine also the existence and stability for coexistence stationary equilibrium. In section \ref{subsection:graphical} we give a full
characterization of the stationary solution and their stability.


Moreover, we can checked (see \cite{sw} appendix B), that the system \eqref{eq:chem_deter} has a  monotonic structure\footnote{  with respect to the order $(x_1,x_2,x_3,x_4)\leq_K (y_1,y_2,y_3,y_4)$ iff 
$x_1\leq y_1$, $x_2\leq y_2$ and $x_3\geq y_3$, $x_4\geq y_4$, see \cite{sw}).}. 
This monotonic structure is a very strong property which reduces the possibilities for the global dynamics of the system. In particular,
for almost every initial condition, the trajectory of the solutions of \eqref{eq:chem_deter} goes to a stationary equilibrium (see \cite{sw},
appendix C). Hence, using the result from the section \ref{section:proofs} and the same arguments that the ones stated in \cite{sw}, 
we obtain theorem \ref{thm:eq2} which describes the possible dynamics of \eqref{eq:chem_deter}. 


\begin{thm}\label{thm:eq2}
Assume that the two semi-trivial stationary equilibrium $E_u$ and $E_v$ exist. 
\begin{enumerate}
 \item If $\Gamma_v>0$ and $\Gamma_u>0$, then the solutions of \eqref{eq:chem_deter} go to the unique coexistence equilibrium $E^*$ which is 
linearly stable for almost every initial condition. 
 \item If $\Gamma_v<0$ and  $\Gamma_u<0$, then there exists an unstable coexistence solution $E_{cu}$. Moreover, the solutions of 
\eqref{eq:chem_deter} go either to $E_u$ of to $E_v$ (for almost every initial condition) depending on the location of the initial 
value according to the basin of attraction of the two semi-trivial equilibrium. We say that it is a situation of exclusive bistability.
 \item Let $\{w,\overline{w}\}=\{u,v\}$ and suppose that $\Gamma_{\overline{w}}<0$ and $\Gamma_w>0$. Then either :
 \begin{enumerate}
  \item There is not coexistence stationary equilibrium. In that case, any solution of \eqref{eq:chem_deter} converges to $E_w$ for almost every initial condition. 
  \item There exist two coexistence stationary equilibrium : one stable $E_{cs}$ and one unstable $E_{cu}$. Any trajectory of 
 \eqref{eq:chem_deter} go either to $E_{cs}$ or to $E_w$ (for almost every initial condition) depending on the location of the 
 initial value according to the basin of attraction of the two stable equilibria. We say that it is a situation of odd bistability.
 \end{enumerate} 
\end{enumerate}
\end{thm}
\begin{rem}As it is proven in \cite{sw},  the cases 2. and 3.b are impossible  
if the consumption functions does not depend on the vessels $\varepsilon^j$.
We show in figure \ref{fig:caspossibles} that every cases may happen in general.
\end{rem}


\begin{figure}
\centering
\begin{tabular}{cc}
a - Typical coexistence case. $R_c$ is associated & b - Typical bi-stable case. $R_c$ is associated to\\
to a globally stable coexistence & an unstable coexistence stationary\\
stationary equilibrium. & equilibrium. $E_u$ and $E_v$ are stable.\\
\includegraphics[scale = 0.3]{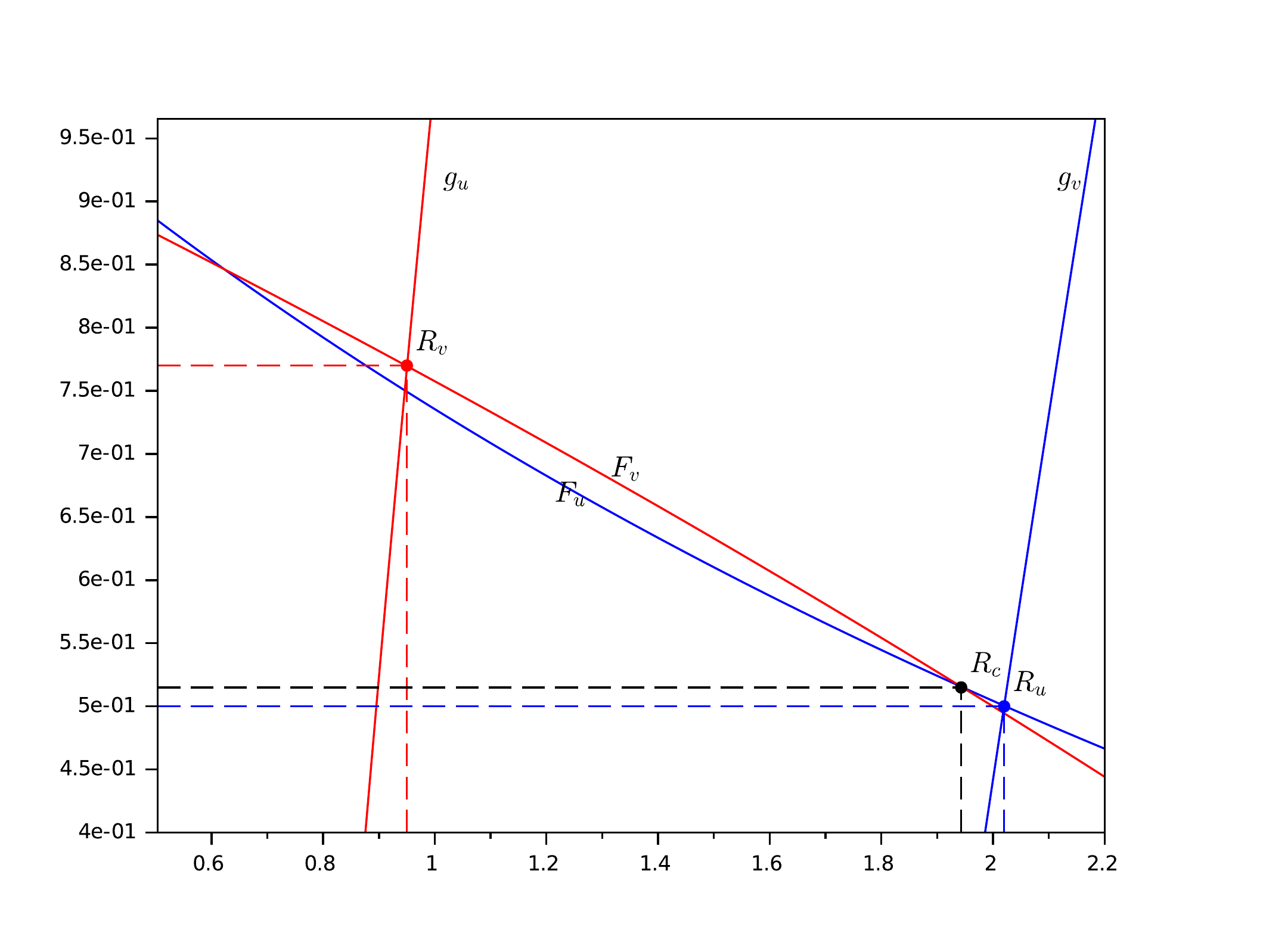} & \includegraphics[scale = 0.3]{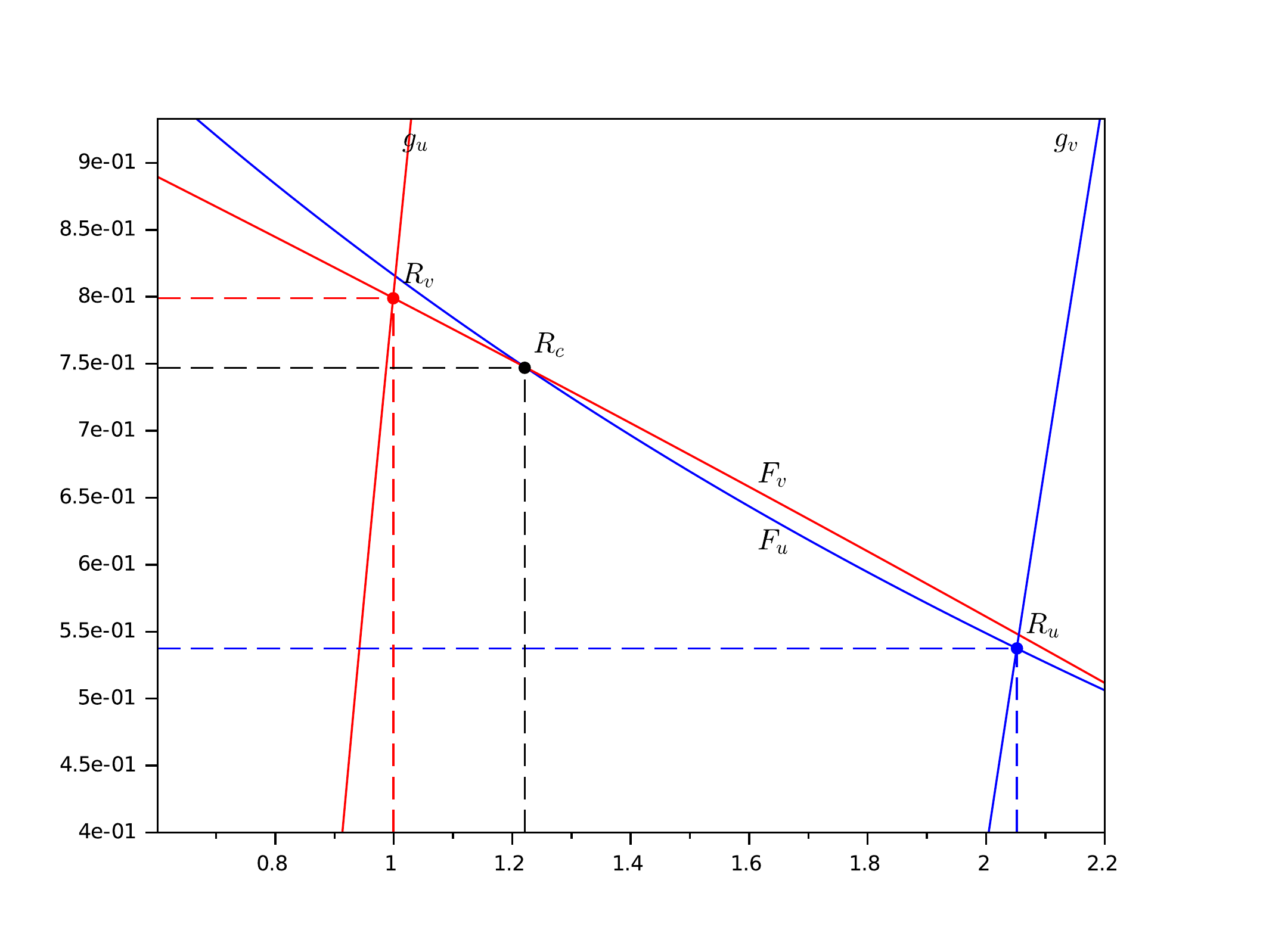}\\
 & d - Rare bi-stable case. $R_{cs}$ is associated \\
 & to a stable equilibrium. $R_{cu}$ is associated to\\
 c - Typical extinction case. Species $u$ & an unstable equilbrium. $E_u$ is stable,\\
 goes to extinction. & $E_v$ is unstable.\\
\includegraphics[scale = 0.3]{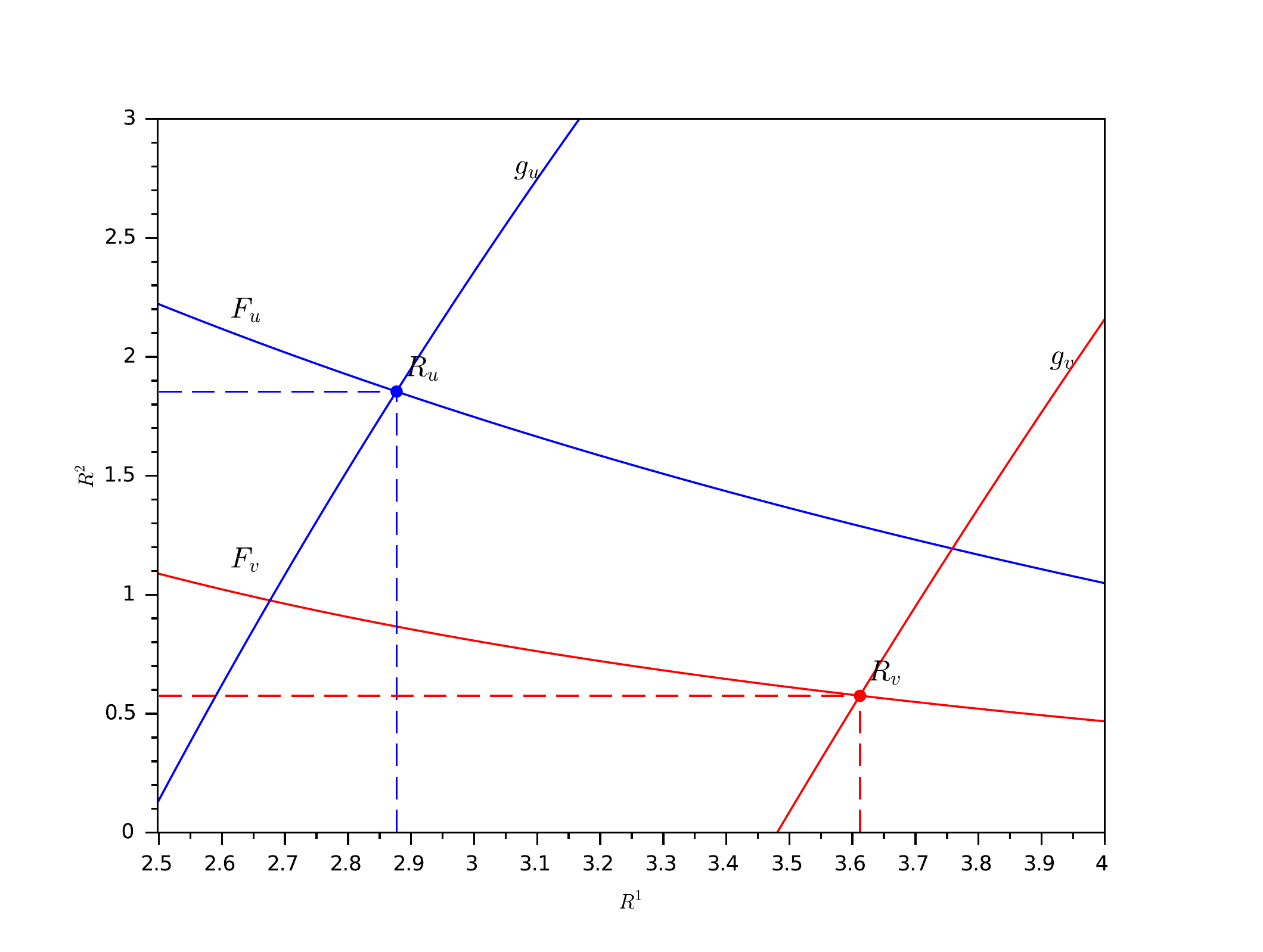} & \includegraphics[scale = 0.3]{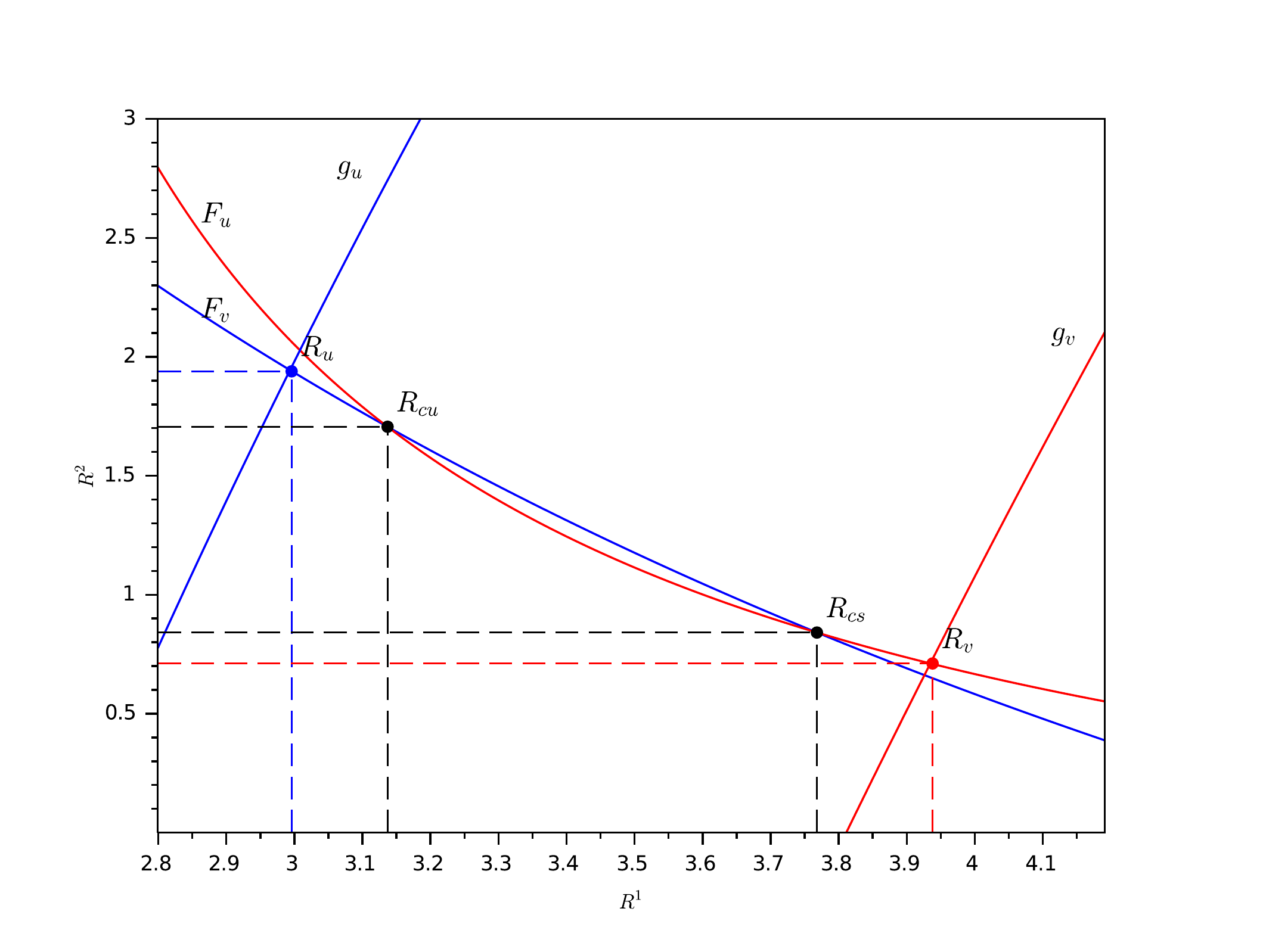}\\
\end{tabular}
\caption{The graph of the functions $F_w$ and $g_w$, $w\in\{u,v\}$ are sufficient to describe the global dynamics of \eqref{eq:chem_deter}. 
The precise definitions of the function $F_w$ and $g_w$ are given in section \ref{section:proofs}.}
\label{fig:caspossibles}
\end{figure}

\section{Comparison of the invasion rates between the two models}
\label{section:ccl}                                                                                  
In section \ref{section:temporal}, a definition for the invasion rates in the probabilistic case is given (equat and it is proven that the signs 
of the invasion rates characterize the long time behavior of the probabilistic model. Recall that in this case, we defined the invasion 
rates by :
\[
 \Lambda_w = \int \left(f_w^{1}(R)-\delta^{1}\right) d\mu_w(R,1) + \int \left(f_w^{2}(R)-\delta^{2}\right) d \mu_w(R,2),
\]
where $\mu$ is an invariant probability measure of $(Z_t)$ on $M_{0,w}$.

In section \ref{section:spatial},  the invasion rates $\Gamma_w$ in the gradostat model are defined as the maximum eigenvalue of certain two dimensional matrices and the theorem \ref{thm:eq2} shows that the sign of these 
invasion rates characterize (essentially)  the behavior of the solutions of the gradostat model. 

In this section, we aim to give a qualitative comparison of the two definition of the invasion rates in order to discuss the similarities
and the differences of the two models we considered.

\subsection{Comparison of the invasion rates in the one species case}
\label{subsection:gammau}
Let us first look at the one species case. The following theorem deals with the probabilistic definition of the invasion rate of 
species $w$.


\begin{thm}\label{thm:inv_prob1}
Let us assume that $R_0^1<R_0^2$ and set $\gamma^j= \frac{\lambda^j}{\delta^j}$. The process $(Z_t)$ has a unique invariant measure when 
it is restricted to $M_{0,w}$. The invasion rate of species $w$ is given by:
\[
\Lambda_w^0 = \frac{\gamma^1+\gamma^2}{\lambda^1+\lambda^2} \dE \left[ \Phi(B) \right].
\]
Where $B$ is a random variable following a Beta law of parameters $(\gamma^1,\gamma^2)$ and:
\[
\Phi(x) = \delta^2(1-x)\left(f^1\left((R_0^2-R_0^1)x+R_0^1\right)-\delta^1\right)+\delta^1 x\left(f^2\left((R_0^2-R_0^1)x+R_0^1\right)-\delta^2\right).
\]
\end{thm}

The unicity of the measure invariant is fairly obvious given the definition of the process $(Z_t)$ restricted to $M_{0,w}$. Its explicit
expression allows to obtain the announced expression for the invasion rate $\Lambda_w^0$. The computation of the invariant measure is postponed
to the last section \ref{proof:inv_prob1} of this article.

Recall that the jump rates of the Markov process $(I_t)$ on the state space $\{1,2\}$ are given by: $\lambda^1 = s\lambda$ and
$\lambda^21 = (1-s)\lambda$ with $\lambda\in\dR$ and $s\in(0,1)$.
\begin{prop}\label{prop:prop_convexity}
The invasion rate $\Lambda_w^0=\left(\frac{s}{\delta^1}+\frac{1-s}{\delta^2}\right)\dE \left[ \Phi(B) \right]$ is monotone acording to the variable $\lambda$.
\end{prop}
Once again the proof of this statement requires heavy computation and is postponed to section \ref{proof:prop_convexity}. This analytical 
property on the invasion rate is used in the proof of theorem \ref{thm:intro_n1}.

An explicit expression of the invasion rate in the deterministic case is given in \eqref{eq:gaga}.
We compute the limits as $\lambda\to0$ and $\lambda\to+\infty$ of these invasion rates.
\begin{prop}
The behavior of the two model is the same when $\lambda$ is large enough.
\[
\underset{\lambda\rightarrow +\infty}{\lim} \Lambda_w^0 = \underset{\lambda\rightarrow +\infty}{\lim} \Gamma_w^0 = (1-s)\left( f_w^1(R^{\infty})-\delta^1\right)+s\left(f_w^2(R^{\infty})-\delta^2 \right)
\]
where $R^{\infty} = \frac{(1-s)\delta^1 R_0^1 + s\delta^2 R_0^2}{(1-s)\delta^1 + s\delta^2}$. \\
The behavior of the two model is not the same the same when $\lambda$ is small enough.
\begin{align*}
&\underset{\lambda\rightarrow 0}{\lim} \Lambda_w^0 = (1-s)\left( f_w^1(R_0^1)-\delta^1\right) + s \left(f_w^2(R_0^2)-\delta^2 \right),\\
&\underset{\lambda\rightarrow 0}{\lim} \Gamma_w^0 = \max\left(f_w^1(R_0^1)-\delta^1, f_w^2(R_0^2)-\delta^2 \right).
\end{align*}
\end{prop}

\begin{rem}\label{rem:agg}
Though these results are easily obtained by a simple computation, the fact that the limits of the invasion rates are the same when 
$\lambda$ goes to $+\infty$ is the consequence of some already known results on the averaging of vector fields. Under some condition over 
the switching vector fields, it is proven in \cite{StrickNaim} that a process built from switching between the different 
vector fields converges in law to the deterministic solution of the aggregated system of the vector fields defined in \ref{rem:chem_moyen}.
\end{rem}

Numerical simulations are presented in \ref{fi:compa1} for two sets of data $\Pi_1$ and $\Pi_2$ defined in the table \ref{tab:data}.

\begin{table}
	\centering
 \begin{tabular}{|c|c|}
  \hline
  $\Pi_1$ & $\Pi_2$ \\
  \hline
  $(a^1,a^2) = (1.1,2)$ & $(a^1,a^2) = (1.1,2)$  \\
  $(b^1,b^2) = (0.4,4)$ & $(b^1,b^2) = (0.05,2)$   \\
  $(\delta^1,\delta^2) = (1,1)$ & $(\delta^1,\delta^2) = (1,1)$  \\
  $(R_0^1,R_0^2) = (10,1)$ & $(R_0^1,R_0^2) = (0.55,2.1)$  \\
  \hline
 \end{tabular}
	\caption{Set of data used in figure \ref{fi:compa1}.}
	\label{tab:data}
\end{table}

\begin{center}

\end{center}


\begin{figure}
\begin{center}
 \begin{tabular}{cc}
  \includegraphics[scale=0.3]{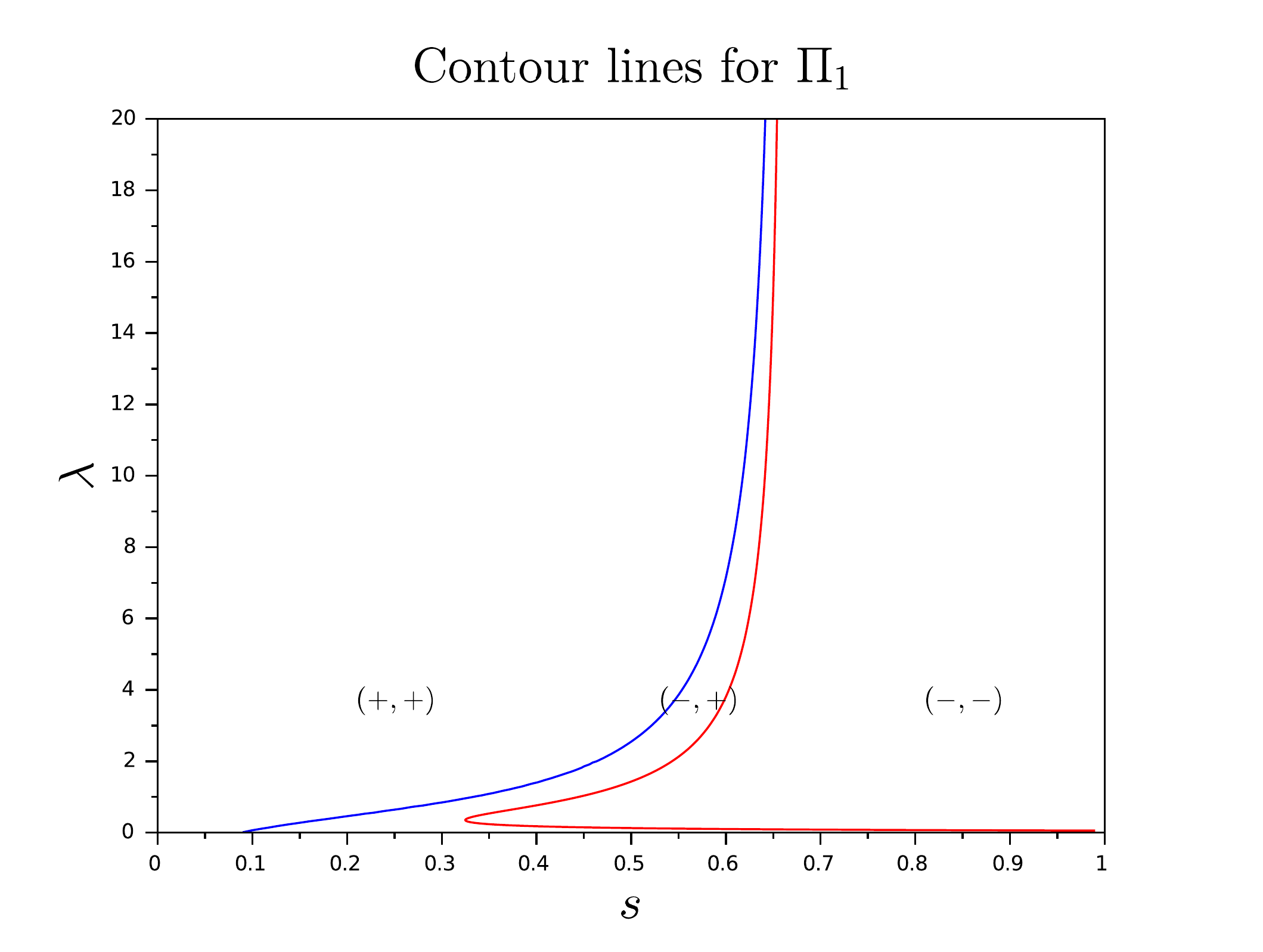} & \includegraphics[scale=0.3]{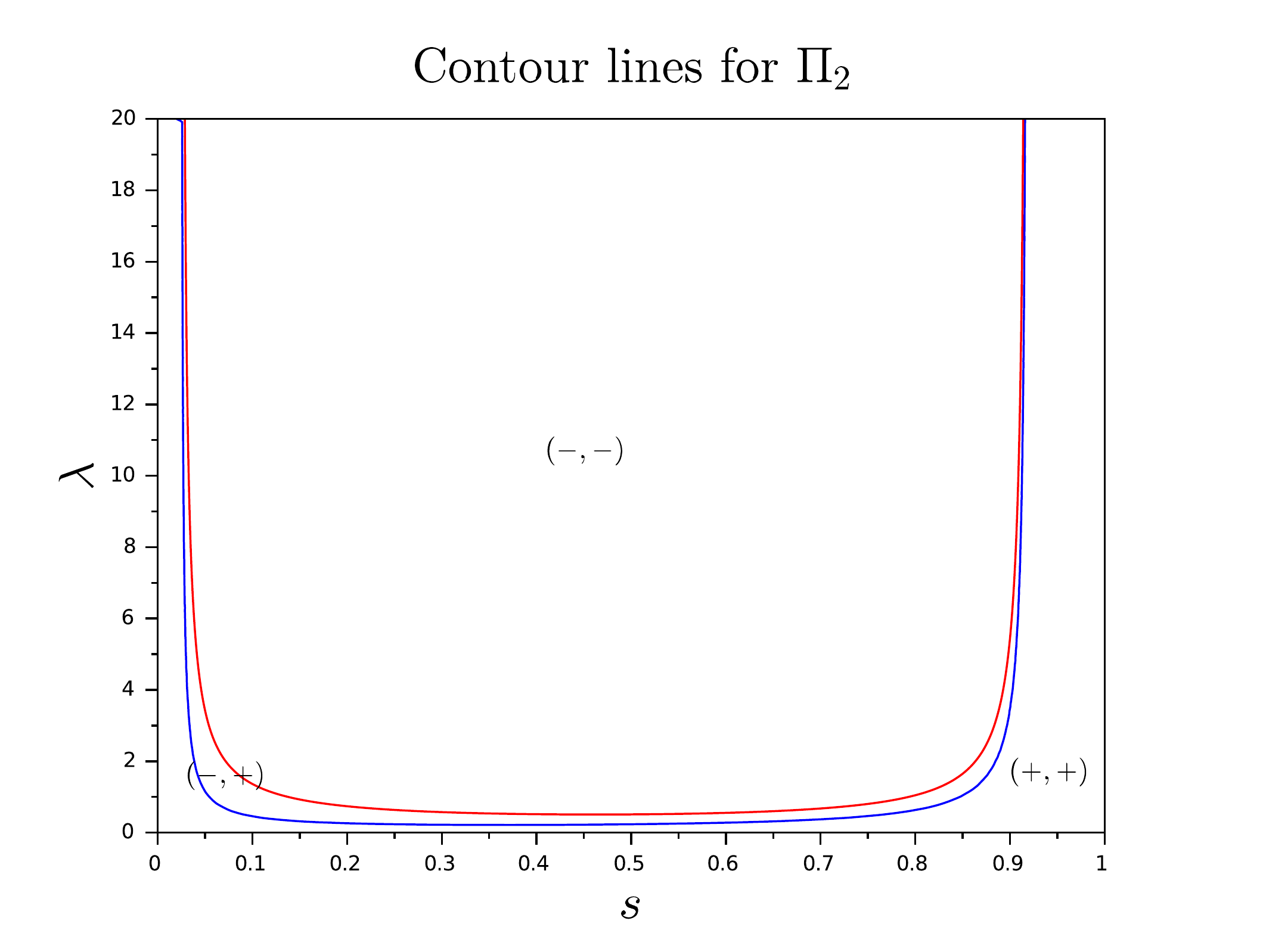}  
 \end{tabular}
 \caption{
Comparisons of the zero level  lines  for $\Gamma^0$ and $\Lambda^0$ for the two 
sets of data $\Pi_1$ and $\Pi_2$. The color blue makes reference to $\Lambda^0$ (probabilistic invasion rate) and the red color 
makes reference to $\Lambda^0$ (deterministic invasion rate).
 In each zone of this figure, the sign of the pair $(\Lambda^0,\Gamma^0)$
is constant and is plainly indicated by a pair of signs.}
 \label{fi:compa1}
\end{center}
\end{figure}



\subsection{Comparison of the invasion rates in the two species case}

We  now have a qualitative discussion on the behavior of the invasion rates when two species are introduced in our models.  Recall that it is assumed here that $R_0^1 = R_0^2$.

\begin{thm}\label{thm:explicitlambda}
The invariant measure $\mu_w$ of $(Z_t)$ restricted to $M_{0,w}$ is unique. The invasion rates $\Lambda_u$ and $\Lambda_v$ are 
computable and there explicit expressions are given by:
\[
 \Lambda_w = \frac{\int h_w(x)g_{\overline{w}}(x)e^{\lambda H_{\overline{w}}(x)}dx}{\int g_{\overline{w}}(x)e^{\lambda H_{\overline{w}}(x)}dx}.
\]
Where:
\[
 h_w(x) = \frac{(f_w^2(R_0-x)-\delta^2)\lvert f_{\overline{w}}^1(R_0-x)-\delta^1 \rvert + (f_w^1(R_0-x)-\delta^1)\lvert f_{\overline{w}}^2(R_0-x)-\delta^2 \rvert}{\lvert f_{\overline{w}}^1(R_0-x)-\delta^1 \rvert +\lvert f_{\overline{w}}^2(R_0-x)-\delta^2 \rvert}
\]
\[
 g_w(x) = \left( \lvert f_w^1(R_0-x)-\delta^1 \rvert + \lvert f_w^2(R_0-x)-\delta^2 \rvert \right)\frac{\lvert f_w^1(R_0-x)-\delta^1 \rvert \lvert f_w^2(R_0-x)-\delta^2 \rvert}{x}
\]
and
\begin{align*}
 H_w(x) = -(\omega_w^1\beta_w^1 + \omega_w^2\beta_w^2)\log(x)&+ \omega_w^1\alpha_w^1 \log\left( (b_w^1+R_0-x)\lvert f_w^1(R_0-x)-\delta^1 \rvert \right)\\
 &+\omega_w^2\alpha_w^2 \log\left( (b_w^2+R_0-x)\lvert f_w^2(R_0-x)-\delta^2 \rvert \right).
\end{align*}
The constants are defined by:
$$
 \gamma_w^j = \frac{\lambda^j}{\delta^j}\frac{R_w^j}{R_0 - R_w^j},\;
 \alpha_w^j = \frac{a_w^j}{a_w^j-\delta^j},\;\beta_w^j = 1 + \frac{R_0}{b_w^j},\;
 \omega_w^1 = \frac{s}{\delta^j}\frac{R_w^1}{R_0 - R_w^1},\; \omega_w^2 = \frac{1-s}{\delta^2}\frac{R_w^j}{R_0 - R_w^2}.$$

\end{thm}
Like for the theorem \ref{thm:inv_prob1}, the proof of this theorem is very computational and is postponed to the last section. This expression for
the probabilistic invasion rate is rather heavy but allows us to do some simulations.

for the deterministic case, the invasion rates $\Gamma_w$ is defined in \ref{defi:gamma} as the maximal eigenvalue of 
 the matrix
$
 M_{\overline{w}}(R_w) 
$
which is defined in \eqref{Mw} and where $R_w$ is the resource concentration at $E_w$. Though it is possible to compute $R_w$ (see section \ref{subsection:graphical}), the complexity of its expressions does not make it interesting to give it formally. However its
explicit expressions is used in the numerical simulations.

\begin{prop}\label{prop:limites}
The behavior of the two models is the same for $\lambda$ large enough.
\[
\underset{\lambda\rightarrow +\infty}{\lim} \Lambda_w = \underset{\lambda\rightarrow +\infty}{\lim} \Gamma_w = (1-s)\left( f_w^1(R_w^{\infty})-\delta^1\right) + s\left(f_w^2(R_w^{\infty})-\delta^2 \right).
\]
where $R_{w}^{\infty}$ is the unique positive solution of the equation:
\[
 (1-s)\left( f_{\overline{w}}^1(R)-\delta^1\right) + s\left( f_{\overline{w}}^2(R)-\delta^2 \right) = 0.
\]
The behavior of the two models is not same for $\lambda$ small enough:
\begin{align*}
&\underset{\lambda\rightarrow 0}{\lim} \Lambda_w = (1-s)\left( f_w^1(R_{\overline{w}}^{1,*})-\delta^1\right) + s\left(f_w^2(R_{\overline{w}}^{2,*})-\delta^2 \right),\\
&\underset{\lambda\rightarrow 0}{\lim} \Gamma_w = \max\left(f_w^1(R_{\overline{w}}^{1,*})-\delta^1, f_w^2(R_{\overline{w}}^{2,*})-\delta^2 \right).
\end{align*}
where
\[
 R_{\overline{w}}^{j,*} =\frac{b_{\overline{w}}^j\delta^j}{a_{\overline{w}}^j-\delta^j} \text{ is the solution of the equation }f_{\overline{w}}^j(R) -\delta^j =0.
\]
\end{prop}

Let us now compare the probabilistic and the deterministic dependance  of the invasion rates with respect to $\lambda$ and $s$ within 
the two models on particular example. In all the following figures, the blue color is associated to the species $u$ whereas the red color 
is associated to the species $v$. The different couple of signs give the couple of signs of the invasion rates $(\Lambda_u,\Lambda_v)$
in the probabilistic case and $(\Gamma_u,\Gamma_v)$ in the deterministic case.
\begin{rem} In all the figure, the zeros level sets of $\Gamma_w$ and $\Lambda_w$ have the same vertical asymptotes since  the two models are described by the {\em same} averaged chemostat $\varepsilon_s$ as $\lambda\to+\infty$ and that $\varepsilon_s$ satisfy the PEC.
\end{rem}
\begin{figure}[h!]
				\hspace{-1.5cm}\begin{tabular}{cc}
		a - Typical coexistence situation.& b - Typical bistability sitation.\\ \includegraphics[scale=0.45]{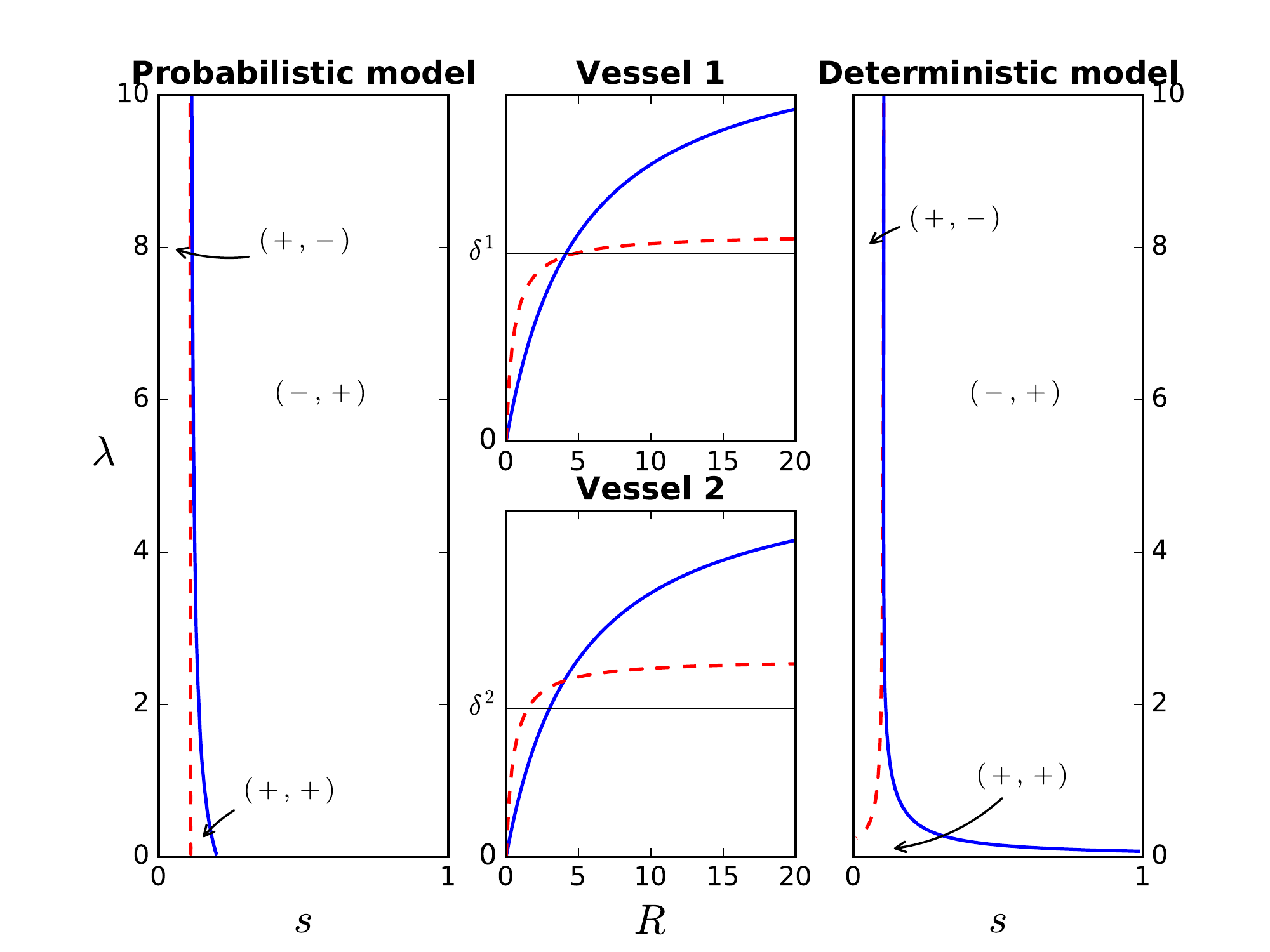}& \includegraphics[scale=0.45]{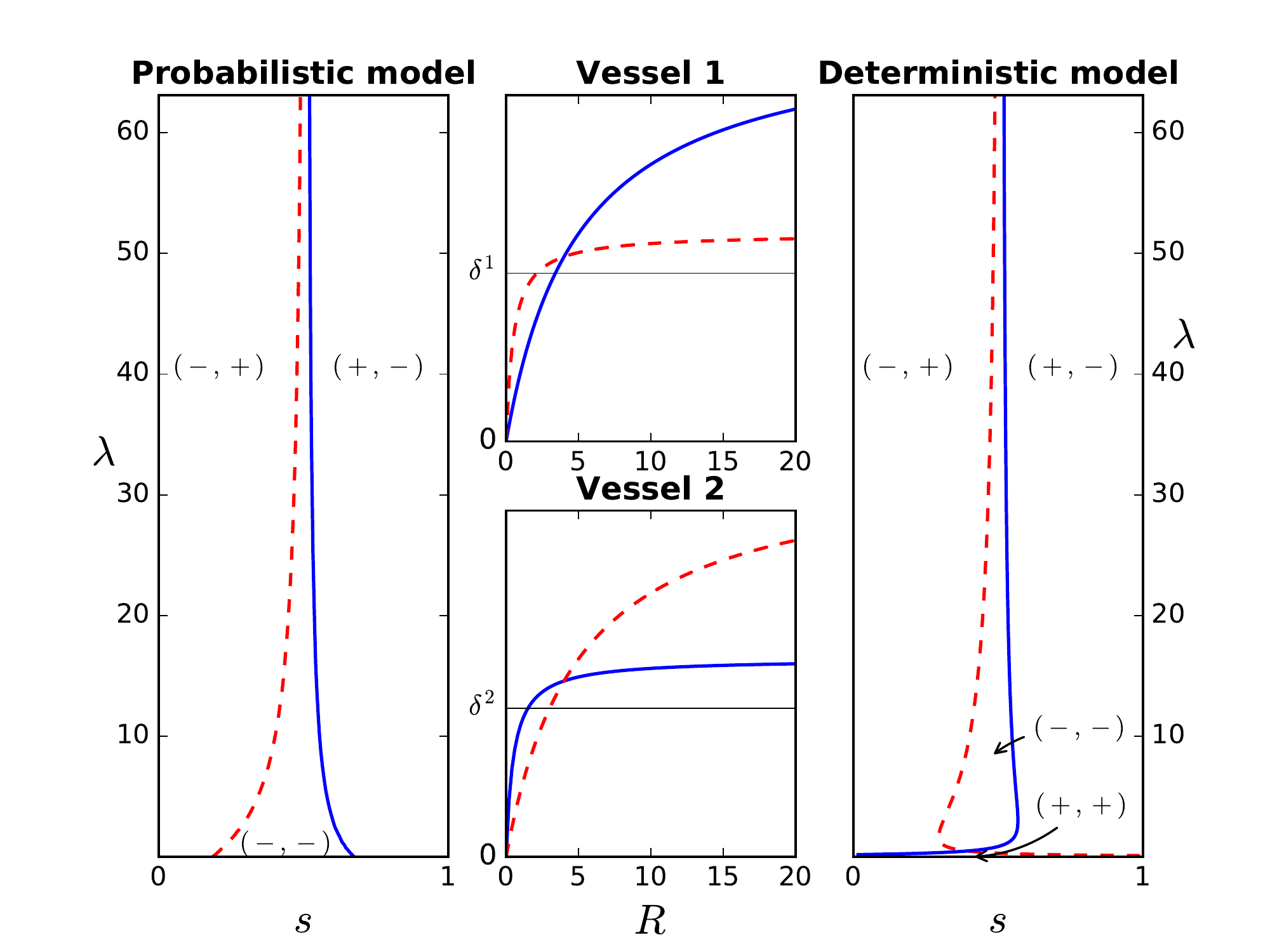}\end{tabular}
				\caption{Both species is the best competor in one vessels.
							a - An appropriate averaged ratio between the vessels leads coexistence $(a_u^1,a_u^2,a_v^1,a_v^2)=(4.2,4,2.1,2)$, $(b_u^1,b_u^2,b_v^1,b_v^2)=(5, 5, 0.5, 0.5)$, 
							$(\delta^1,\delta^2)=( 1.9,1.5)$ and $R_0=8$.
							b - The role of species are reversed between the vessels. For the probabilistic model, there is either exclusion or bistability. The same holds for the deterministic case, exept that small diffusion permits coexistence.
								$(a_u^1,a_u^2,a_v^1,a_v^2)=(4.2,2,2.1,4)$, $(b_u^1,b_u^2,b_v^1,b_v^2)=(5, 0.5, 0.5, 5)$, 
							$(\delta^1,\delta^2)=( 1.7,1.5)$ and $R_0=8$.
						}
						\label{fig:fig1}
\end{figure}

				\begin{figure}[h!]
				\hspace{-1.8cm}\begin{tabular}{cc}
		a - Two vessels favorable to the  species $u$.& b -Odd bistability  in the deterministic model.\\
		\includegraphics[scale=0.40]{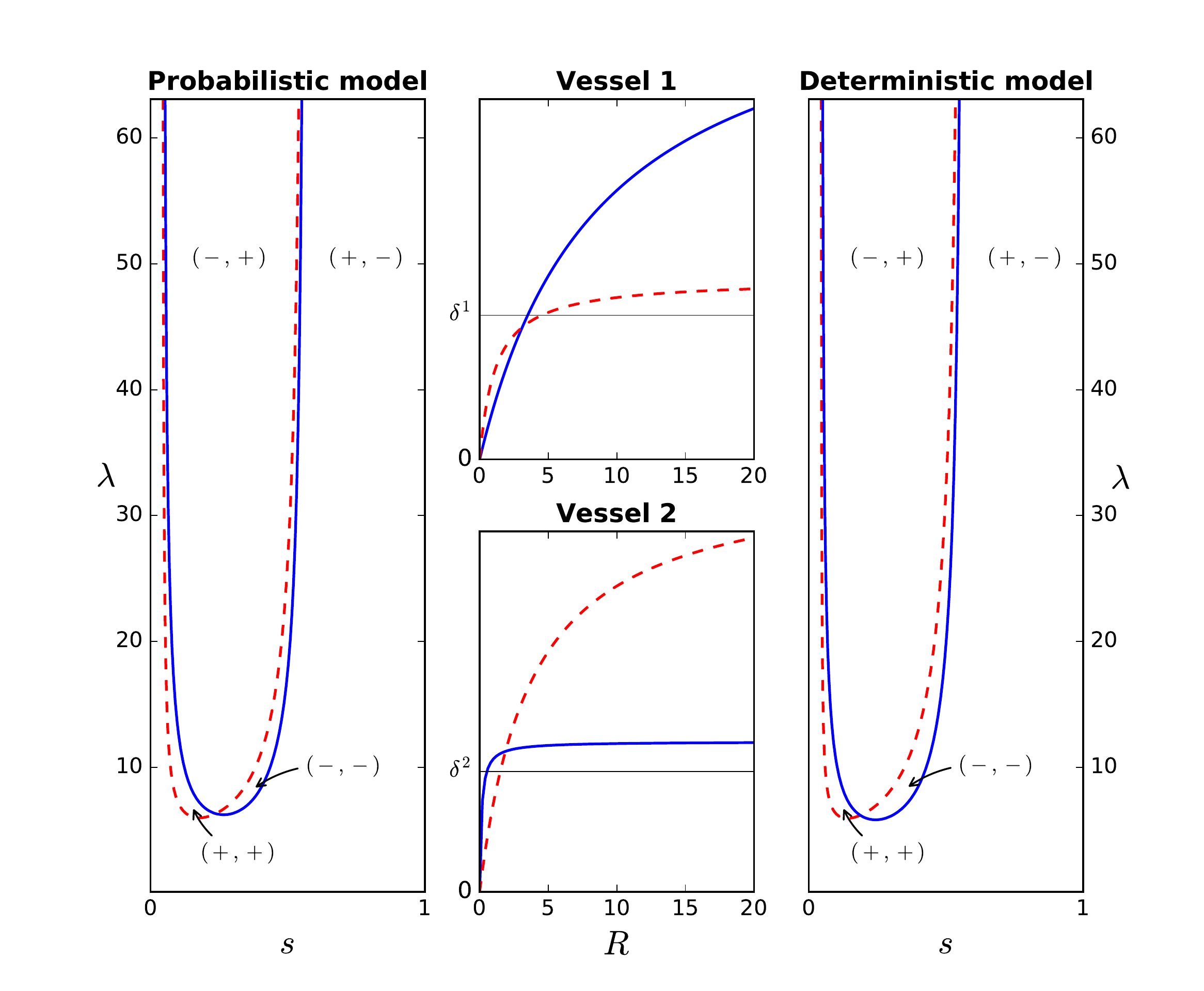}& 
		\hspace{-1.37cm}\includegraphics[scale=0.35]{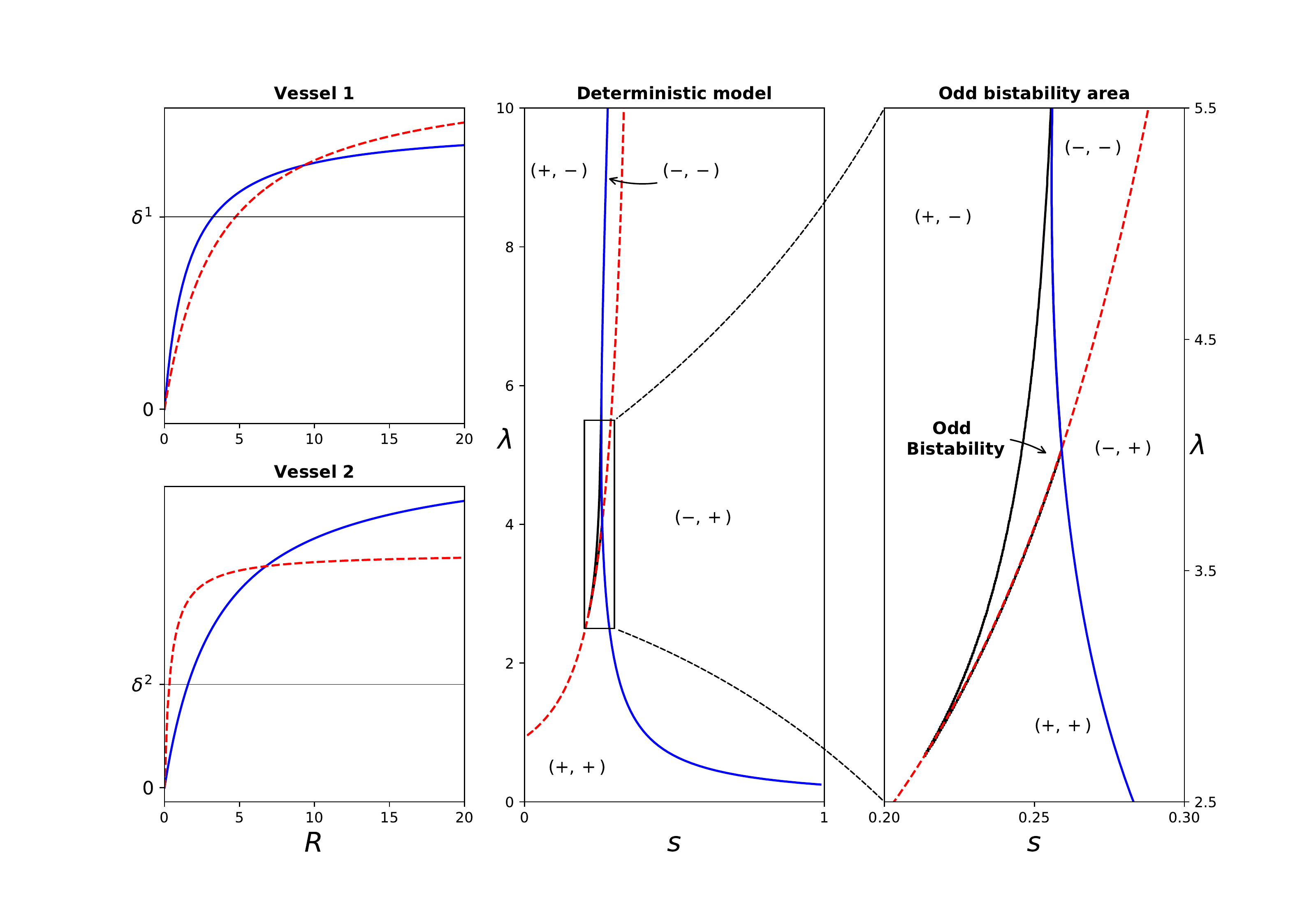}
		\end{tabular}
				\caption{Two interesting situations.
				a - The two vessels are favorable to the same species. Depending on $\lambda$ and $s$, each situation may occurs for both models (extinction of $u$ or $v$, exclusive bistability or coexistence). $(a_u^1,a_u^2,a_v^1,a_v^2)=(3.5,2.5,1.25,7)$,
  $(b_u^1,b_u^2,b_v^1,b_v^2)=(8.75, 0.125, 1.125, 3.75)$, $(\delta^1,\delta^2)=( 1,2)$ and $R_0 = 7$. 
							b - A situation like in figure \ref{fig:fig1}-a with an odd bistable area in the deterministic model (the probabilistic model behaves like the one figure \ref{fig:fig1}-a). We show only the deterministic model and make a zoom on the odd bistable area in a $(+,-)$ area. This zone corresponds to the case 3-(b) in the theorem \ref{thm:eq2}.
   $(a_u^1,a_u^2,a_v^1,a_v^2)=(3.7,3.6,4.4,2.5)$, 
	$(b_u^1,b_u^2,b_v^1,b_v^2)=(1.55, 3.55, 3.6, 0.4)$,
	$(\delta^1,\delta^2)=( 2.5,1.1)$ and $R_0=20$.
	}
\label{fig:fig2}
\end{figure}


%

\begin{rem}\label{rem:mono2}
Numerically, the invasion rates $\Lambda_w$ seem to have a monotonous behavior according to $\lambda$ just like in the case $n=1$. Sadly the complexity
of their expressions does not allow us to prove it. We will conjecture it. 
Under this conjecture, we do not need the assumption $H_w$ in the theorem \ref{thm:intro_n2}.\\
Ours numerical examples shows that this is not the cases for the deterministic model, even for $n=1$ (see figure \ref{fi:compa1}-a, \ref{fig:fig1}-b and \ref{fig:fig2}-b).
\end{rem}

\subsection{Concluding remarks}\label{subsection:cclccl}
Let us conclude on the similarities and differencies between the two models we studied in this chapter. For each models we gave a 
definition of the invasion rates of the introduced species which depend only on the parameters of the systems. Despite the differences of 
their mathematical nature, theorem \ref{thm:intro_n2} and \ref{thm:eq2} show that the long-time behaviors of the two models essentially depend 
on the signs of the invasion rates. Hence, we compared the two models by comparing the behavior of the invasion rates according 
to the parameters $(s,\lambda)$ (where $\lambda^1 = s\lambda$ and $\lambda^2 = (1-s)\lambda$). 
In the probabilistic case, $(\lambda^1,\lambda^2)$ are the parameters of the 
Markov chain governing the switching between the environments whereas in the deterministic case, $(\lambda^1,\lambda^2)$ are the exchange 
parameters between the two vessels. 

From the previous theorems and numerical simulations come the following similarities between the two models:
\begin{itemize}
 \item When the invasion rates are positive (resp. negative) for $u$ and $v$, the probabilistic system and the deterministic system
are in a coexistence state (resp. bistable state). Moreover, we proved numerically that it is possible to have bistability with two introduced
species and two vessels. This numerical result is similar to the result of \cite{grad5} where they proved in their particular case (dilutions rates and 
consumption functions not depending on the vessel, two introduced species) that at least three vessels are needed for the existence of an 
unstable coexistence equilbrium.
 \item The limits of the invasion rates when $\lambda$ goes to infinity are the same for both models. We saw that the reason behind this 
result is the averaging phenomenon occuring when $\lambda$ is large enough implying that both systems behave like the averaged chemostat
$\varepsilon_s$. Graphically, we see that the zero contour lines of the invasion rates are really alike for $\lambda$ large enough and 
have the same asymptote when $\lambda$ goes to infinity.
\end{itemize}
The main differences between our competition models are the following:
\begin{itemize}
 \item In the probabilistic model, when the invasion rates have opposite signs, only one species survives, the one with the positive 
invasion rate. However, in the deterministic model, when the invasion rates have opposite signs, it is possible for the system to be 
in an ``odd'' bistable state where one of the stable stationary equilibrium is a coexistence equilibrium an the other a semi-trivial solution.
 \item The most important difference between the two models occurs when $\lambda$ is close to zero because the limits of the invasion rates
when $\lambda$ goes to zero are different. We can interpret this difference by the difference of nature between the two
models when $\lambda$ is very small. For the probabilistic model, $\lambda$ very small implies that the process follows for a very long
time the flow of each chemostat $\varepsilon^1$ and $\varepsilon^2$ and the invasion rates measures the averaging of the behavior 
of each flows. But in the deterministic case, when $\lambda$ is very small, there are almost no exchanges between the two vessels implying 
that the system almost behaves like two isolated chemostats with a very small diffusion between them.
\end{itemize}

We give here a little discussion over the parameter restrictions we did on our models. First, note that the most important parameters 
involved in the heterogenity of our two models are the quantities $R_w^j$ which are the minimum resource quantities needed by species $w$ 
to survive in the vessel $j$ (when the vessels are isolated). Recall that $R_w^j$ is solution of the equation:
\[
 f_w^j(R) - \delta_w^j = 0
\]
where $f_w^j$ are the consumption functions and $\delta_w^j$ the dilution rates. As a consequence, allowing the consumption
functions or the dilution rates to depend on $w$ and $j$ is the easiest way to allow the parameters $R_w^j$ to be different according to 
$w$ and $j$.

Note that in the probabilistic model we had to assume that the ressource entries $R_0^j$ are equal in order to reduce the system and do 
some computations. But this hypothesis is not necessary in the deterministic model where we claim that the computations 
are still possible. In fact, in \cite{sw}, the authors model the environment heterogeneity with a different resource input for each vessel,
and thanks to this heterogeneity, a coexistence stationary equilibrium may appear. In our case, we model the environment heterogeneity by 
taking vessel dependant consumption functions and dilution rates.

In this paper,
we decided that only the consumption functions will depend on $w$ and $j$ while the dilution rates only depend on the vessel $j$. This
hypothesis is crucial because it allows us to reduce the systems of differential equations (thanks to the variable $\Sigma$) into a 
monotonous system, ultimately leading to the long-time behavior theorems. However, it was not a natural choice in the deterministic 
model because in the gradostat applications, the consumption functions do not depend on the vessels but only on the species. As 
a consequence, this hypothesis took us away from the gradostat context (and its application in the industry for example) to bring us in a 
more theoretical ecological study of the spatial heterogeneity.

Nonetheless, the approach with the functions $F_w$ and $g_w$ might lead to the obtention of the existence and stability of the 
stationary equilibria of the gradostat-like model when the dilution rates also depend on the species and can be the subject of some future
work.

\section{Mathematical proofs}\label{section:proofs}

\subsection{Computation of the invariant measures in the probabilistic case}

We show in this subsection how to compute the invariant measures announced in theorem \ref{thm:inv_prob1} and \ref{thm:explicitlambda}.

\subsubsection{Proof of the theorem \ref{thm:inv_prob1}}\label{proof:inv_prob1}


\begin{proof}
Recall that only one species is introduced in our system. The invasion rate $\Lambda_w^0$ is defined by:
\[
 \Lambda_w^0 = \int \left(f_w^{1}(R)-\delta^{1}\right) d\mu_w^0(R,1) + \int \left(f_w^{2}(R)-\delta^{2}\right) d \mu_w^0(R,2)
\]
where $\mu_w^0$ is an invariante measure of the process $(Z_t)$ restricted $M_{0,w}$. On $M_{0,w}$, $(Z_t)=(R_t,0)$ satisfies:
\[
 \dot{R_t} = \delta^{I_t}(R_0^{I_t}-R_t).
\]
Its infinitesimal generator is given for any good functions $f$ by:
\[
 Lf(r,i) = \delta^i(R_0^i-r)f'(r,i) + \lambda^1(f(r,\overline{i})-f(r,i)).
\]
It is clear that for $t$ large enough, $(R_t)$ belongs to $[R_0^1,R_0^2]$. By compacity, there exists an invariant measure for $(R_t)$ and
it is unique because the process is recurrent.

The unique invariant measure $\mu_w^0$ satisfies:
\begin{equation}\label{eq:mes_wnv}
 \forall f,\quad \int Lf(r,i) d\mu_w^0 = 0.
\end{equation}
We search $\mu_w^0$ of the shape $\mu_w^0(dR,j) = \rho^j(R) \ind_{j}dR$.
 It gives in \ref{eq:mes_wnv}:
\begin{equation}\label{eq:dirty}
\begin{aligned}
  &\int_{R_0^1}^{R_0^2} \left( \delta^1 (R_0^1 - R)f'(R) + \lambda^1(f(R,2)-f(R,1)) \right)\rho^1(R)dR + \\
  &\int_{R_0^1}^{R_0^2} \left( \delta^2 (R_0^2 - R)f'(R) + \lambda^2(f(R,1)-f(R,2)) \right)\rho^2(R)dR =0.
\end{aligned}
\end{equation}
Assume that $f(x,j) = f(x)$. It gives in \ref{eq:dirty}:
\[
\int_{R_0^1}^{R_0^2} \left( \delta^1 (R_0^1 - R)f'(R) \right)\rho^1(R)dR + 
  \int_{R_0^1}^{R_0^2} \left( \delta^2 (R_0^2 - R)f'(R) \right)\rho^2(R)dR =0.
\]
An integration by parts gives:
\[
\begin{aligned}
 \left[ \delta^1 (R_0^1 - R)f'(R)\rho^1(R) \right]_{R_0^1}^{R_0^2}&+\left[ \delta^2 (R_0^2 - R)f'(R)\rho^2(R) \right]_{R_0^1}^{R_0^2}\\
 &-\int_{R_0^1}^{R_0^2} f(x)\left( (\delta^1 (R_0^1 - R)\rho^1(R))' + (\delta^2 (R_0^2 - R)\rho^2(R))'\right)dR = 0.
\end{aligned}
\]
it seems ``natural'' that $\rho^j(R_0^{\overline{j}}) = 0$ according to the dynamics of the process $(R_t,I_t)$. Assuming this, a classic density
argument gives:
\[
 \delta^1 (R_0^1 - R)\rho^1(R) + \delta^2 (R^2 - R)\rho^2(R) = K.
\]
 From $\rho^j(R_0^{\overline{j}}) = 0$ we have $K=0$ which yields:
\begin{equation}\label{eq:eq1}
\delta^1 (R_0^1 - R)\rho^1(R) + \delta^2 (R^2 - R)\rho^2(R) = 0.
\end{equation}
Now, assume that $f(R,1)=f(R)$ and $f(R,2)=0$. Plugging this  in \ref{eq:dirty} gives
after an integration by parts:
\[
 \int_{R_0^1}^{R_0^2} f(R) \left( -(\delta^1 (R_0^1 - R)\rho^1(R))'-\lambda^1\rho^1(R)+\lambda^2 \rho^2(R) \right)dR.
\]
By the same density argument as before, we obtain
\[
 -(\delta^1 (R_0^1 - R)\rho^1(R))'-\lambda^1\rho^1(R)+\lambda^2 \rho^2(R) = 0
\]
that is
\[
 -\delta^1 (R_0^1 - R)\rho'^1(R) + \delta^1\rho^1(R) -\lambda^1\rho^1(R)+\lambda^2 \rho^2(R) = 0.
\]
Equation \ref{eq:eq1} gives:
\[
 \rho^2(R) = \frac{\delta^1 (R-R_0^1)}{\delta^2 (R^2-R)}\rho^1(R).
\]
As a consequence, $\rho^1$ satisfies the differential equation:
\begin{equation}\label{eq:eq2}
 \rho'^1(R) + \rho^1(R) \left( \frac{1}{R-R_0^1}-\frac{\lambda^1}{\delta^1(R-R_0^1)}+\frac{\lambda^2}{\delta^2(R_0^2-R)} \right) =0.
\end{equation}
Solving \ref{eq:eq2} gives the explicit expression for $\rho^1$:
\[
 \rho^1(R) = C(R-R_0^1)^{\frac{\lambda^1}{\delta^1}-1}(R_0^2-R)^{\frac{\lambda^2}{\delta^2}}.
\]
Hence,
\[
 \rho^2(R) = C\frac{\delta^1}{\delta^2}(R-R_0^1)^{\frac{\lambda^1}{\delta^1}}(R_0^2-R)^{\frac{\lambda^2}{\delta^2}-1},
\]
where $C$ is a constant. The value of $C$ is determined by the fact that $\mu$ is a probability measure:
\[
 \int_{R_0^1}^{R_0^2} \rho^1(R)dR + \int_{R_0^1}^{R_0^2} \rho^2(R)dR = 1.
\]
As a consequence:
\[
 C\int_{R_0^1}^{R_0^2} \left( (R-R_0^1)^{\frac{\lambda^1}{\delta^1}-1}(R_0^2-R)^{\frac{\lambda^2}{\delta^2}} + \frac{\delta^1}{\delta^2}(R-R_0^1)^{\frac{\lambda^1}{\delta^1}}(R_0^2-R)^{\frac{\lambda^2}{\delta^2}-1} \right)dR = 1.
\]
This explicit expression of  $\mu_w^0$ allows us to compute $\Lambda_w^0$:
\[
 \Lambda_w^0 = C\delta^2\int_{R_0^1}^{R_0^2} (f_w^1(R)-\delta^1)(R-R_0^1)^{\frac{\lambda^1}{\delta^1}-1}(R_0^2-R)^{\frac{\lambda^2}{\delta^2}}dR
 + C\delta^1 \int_{R_0^1}^{R_0^2} (f_w^2(R)-\delta^2) (R-R_0^1)^{\frac{\lambda^1}{\delta^1}}(R_0^2-R)^{\frac{\lambda^2}{\delta^2}-1}dR
\]
Set  $x = \frac{R-R_0^1}{R_0^2-R_0^1}$, $\gamma^j = \frac{\lambda^j}{\delta^j}$ and
$g_w^j(x) = f_w^j((R_0^2-R_0^1)x+R_0^1)$, we obtain
\[
 \Lambda_w^0 = C (R_0^2-R_0^1)^{\gamma^1+\gamma^2} \int_0^1 \left[\delta^2(g_w^1(x)-\delta^1)(1-x)+\delta^1(g_w^2(x)-\delta^2)x \right] x^{\gamma^1-1}(1-x)^{\gamma^2-1}dx
\]
One can recognize a part of the density of the Beta law of parameters $(\gamma^1,\gamma^2)$.
Using the same variable change for the expression of $C$ and some classical properties of the beta function (like $B(x,y) = B(y,x)$ and
$B(x,y+1) = \frac{y}{x+y}B(x,y)$), the expression of $\Lambda$ becomes:
\[
 \Lambda_w^0 = \frac{\gamma^1+\gamma^2}{\lambda^1+\lambda^2} \int_0^1 \left[\delta^2(g_w^1(x)-\delta^1)(1-x)+\delta^1(g_w^2(x)-\delta^2)x \right] \frac{x^{\gamma^1-1}(1-x)^{\gamma^2-1}}{B(\gamma^1,\gamma^2)}dx
\]
Set $\Phi(x) = \delta^2(g_w^1(x)-\delta^1)(1-x)+\delta^1(g_w^2(x)-\delta^2)x$, then:
\begin{equation}\label{eq:L}
 \Lambda_w^0 = \frac{\gamma^1+\gamma^2}{\lambda^1+\lambda^2} \dE \left[ \Phi(B) \right]
\end{equation}
where $B$ is a random variable following a Beta law of parameter $(\gamma^1,\gamma^2)$.
\end{proof}

\begin{rem}
The proof for theorem \ref{thm:explicitlambda} uses the same idea except that it requires more heavy computations. We omit it for the sake
of readability of this article.
\end{rem}

\subsubsection{Proof of the proposition \ref{prop:prop_convexity}}\label{proof:prop_convexity}

Our expression of the invasion rate is similar to the one the authors of \cite{mz} obtained for the invasion rates defined in the 
Lotka-Volterra switching system introduced in \cite{lv}. In order to study the invasion rate they use the following property:

\begin{prop}\label{prop:conv}
(Convex order between Beta laws). Assume that $X$ and $X'$ are two random variables following Beta laws of parameters $(a,b)$ and $(a',b')$. If
$a<a'$, $b<b'$ and $\frac{a}{a+b} = \frac{a'}{a'+b'}$ then for any convex function $\phi$:
\[
 \dE[\phi(X')] \leq \dE[\phi(X)].
\]
\end{prop}

We will use this proposition in order to prove the following proposition:

\begin{prop}
The invasion rate $\Lambda_w^0$ is monotone acording to the variable $\lambda$.
\end{prop}

\begin{proof}
We proved that:
\[
  \Lambda_w^0 = \frac{\gamma^1+\gamma^2}{\lambda^1+\lambda^2} \dE \left[ \Phi(B) \right].
\]
Recall that $\gamma^1(s,\lambda) = \frac{s\lambda}{\delta^1}$ and $\gamma^2(s,\lambda) = \frac{(1-s)\lambda}{\delta^2}$. Proposition \ref{prop:conv} 
ensures that if $B$ and $B'$ are random variables following Beta law of parameters $(\gamma^1(s,\lambda),\gamma^2(s,\lambda))$ and
$(\gamma^1(s,\lambda'),\gamma^2(s,\lambda'))$ with $\lambda<\lambda'$ then for any convex function $\phi$:
\[
 \dE[\phi(B')] \leq \dE[\phi(B)].
\]
As a consequence, establishing the convexity (or concavity) of the function $\Phi$ can give the monotonicity of $\Lambda$ according to
the global switching rate $\lambda$.

Recall that:
\[
\Phi(x) = \delta^2(1-x)\left(f_w^1\left((R_0^2-R_0^1)x+R_0^1\right)-\delta^1\right)+\delta^1 x\left(f_w^2\left((R_0^2-R_0^1)x+R_0^1\right)-\delta^2\right).
\]
Here the convexity (or concavity) of $\Phi$ is not clear and will be checked by straight computation.
Set $\alpha^j = \frac{a_w^j}{\delta^j}$, $\beta^j = \frac{b_w^j}{R_0^2-R_0^1}$ and $r = \frac{R_0^1}{R_0^2-R_0^1}$. It comes:
\[
\Phi(x) = \delta^1 \delta^2 \left( (1-x)\left( \frac{\alpha^1(x+r)}{\beta^1+x+r}-1 \right)+ x\left( \frac{\alpha^2(x+r)}{\beta^2+x+r}-1 \right) \right).
\]
Set $t=x+r$ ($t\in[r,1+r]$). It comes:
\[
 g(t) = \frac{\Phi(t)}{\delta^1 \delta^2} = (1+r-t)\left( \frac{\alpha^1 t}{\beta^1 + t}-1 \right) + (t-r)\left( \frac{\alpha^2 t}{\beta^2 + t}-1 \right).
\]
A straight forward computation gives the derivatives of $g$:
\[
g'(t) = (1+r-t)\frac{\alpha^1\beta^1}{(t+\beta^1)^2}-\frac{\alpha^1 t}{\beta^1 + t} +(t-r)\frac{\alpha^2\beta^2}{(t+\beta^2)^2} +\frac{\alpha^2 t}{\beta^2 + t}
\]
and
\[
 \frac{g''(t)}{2} = \frac{-\alpha^1\beta^1(1+r+\beta^1)(t+\beta^1)^3+\alpha^2\beta^2(r+\beta^2)(t+\beta^1)^3}{(t+\beta^1)^3(t+\beta^2)^3}.
\]
Set $L^1 = \alpha^1\beta^1(1+r+\beta^1)$ and $L^2 = \alpha^2\beta^2$. It comes:
\begin{align*}
 h(t) &= \frac{g''(t)}{2}(t+\beta^1)^3(t+\beta^2)^3\\
 &= (L^2-L^1)t^3 + 3\left(\beta^1 L^2-\beta^2 L^1\right) t^2 +3\left((\beta^1)^2 L^2-(\beta^2)^2 L^1\right)t + (\beta^1)^3 L^2-(\beta^2)^3 L^1.
\end{align*}
Set $L = \frac{L^2}{L^1}$ and $\beta = \frac{\beta^1}{\beta^2}$, it comes:
\[
 h(t) = (\beta^2)^3\left( (L-1)\left(\frac{t}{\beta}\right)^3 + 3(L\beta-1)\left(\frac{t}{\beta}\right)^2 + 3(L(\beta)^2-1)\left(\frac{t}{\beta}\right) +L(\beta)^3-1 \right).
\]
The study of the polynomial $P = (L-1)X^3 + 3(L\beta-1)X^2 + 3(L(\beta)^2-1)X +L(\beta)^3-1$ will give the sign of the second derivative of $\Phi$.

\begin{lem}
$P$ has a unique root on $\dR$ and its expression is:
\[
X_0 = \left\lvert \frac{\beta-1}{L-1} \right\rvert \left( -L^{\frac{1}{3}}-L^{\frac{2}{3}} \right) - \frac{L\beta-1}{L-1}.
\]
Moreover, $X_0<0$.
\end{lem}
\begin{proof}
This result is proven by a computation of the roots of the polynomial $P$. It comes that $P$ has a unique root and it is negative.
%
%
%
%
%
\end{proof}

It comes from this previous lemma that the second derivative of $\Phi$ has a constant sign on $[0,1]$ implying that $\Phi$ is either
convex or concave on $[0,1]$. So $\Lambda_w^0$ is monotonous according to \ref{prop:conv}.
\end{proof}

\subsection{Proof of the results for the deterministic model}

\subsubsection{A graphical caracterisation of the equilibria and their stability}%
\label{subsection:graphical}

In this section, we construct a graphical approach in the plan $(R^1,R^2)$ which contains all the information about the non negative 
stationary solution and their stability. This approach is based on the construction of four functions $F_w$ and $g_w$, $w\in\{u,v\}$ 
described below.

For the sake of simplicity we set 
\begin{equation}\label{eq:Xwj}
 X_w^j(R^j) = f_w^j(R^j)-\delta^j.
\end{equation}

Any non-negative stationary equilibrium $(U,V)$ of the differential equation \eqref{eq:chem_deter} are  solution of the 
system \eqref{eq:station1}:
\begin{equation}\label{eq:station2}
\left\{
\begin{aligned}
 &A_u(R)U = 0\\
 &A_v(R) V = 0
\end{aligned}
\right.
\end{equation}
where, according to remark \ref{rem:chgt_var}, we have $R=R_0-U-V\in[0,R_0]$ and the matrices $A_w(R)$ are defined by
\[
 A_w(R) = \begin{pmatrix}
        X_w^1(R)-\lambda^1 & \lambda^1\\
        \lambda^2 & X_w^2(R)-\lambda^2
       \end{pmatrix}.
\]



Recall that for any $w\in\{u,v\}$, we note $W\in\{U,V\}$ the concentration of the species $w$. If $W\neq\begin{pmatrix}
                                                                                                         0\\
                                                                                                         0
                                                                                                        \end{pmatrix}$ in \eqref{eq:station2}, 
it implies that $\det\left(A_w(R)\right)=0$ which reads explicitly:
\begin{equation}\label{definitionFintext}
 \left(X_w^1(R^1)-\lambda^1\right)\left(X_w^2(R^2) -\lambda^2\right)=\lambda^1\lambda^2.
\end{equation}
It follows that the set of points $(R^1,R^2)$ for which the species $w$ may survive is a one dimensional curve. It appears 
that this curve is the graph of a decreasing function $F_w$ defined on a domain $D_w$: 
\[
 (R^1,R^2)\text{ verifies \eqref{definitionFintext}}\; \Leftrightarrow\;R^1\in D_w \text{ and } R^2=F_w(R^1).
\]
Moreover, these functions $F_w$ may be explicitly computed as it is stated in the  proposition \ref{prop:Fw}.
\begin{prop}\label{prop:Fw}
Let $w\in\{u,v\}$ and $g:\;x\mapsto g(x)=\lambda^2+\frac{\lambda^1\lambda^2}{x-\lambda^1}.$
Define: 
\[
 D_w=\{r\in[0,R_0],\; X_w^1(r)-\lambda^1<0\}\text{ and }
 F_w= \left(X_w^2\right)^{-1} \circ g \circ X_w^1. 
\]

Now, suppose that there  exists a non-negative solution $(U,V)$ of \eqref{eq:station2} such that $W\in\{U,V\}$ is non zero. Then
\[
 R^1\in D_w \text{ and } R^2=F_w(R^1)
\]
\end{prop}
\begin{rem}\label{rem:F}
The functions $X_w^j$ being increasing and the function $g$ being decreasing, the identity $X_w^2\circ F_w=g\circ X_w^1$ implies that the functions $F_w$ are strictly decreasing on their definition set. Moreover it exists $(m_w^1,m_w^2,m_w^3,m_w^4)\in\dR^4$ such that:
\[
 F_w(x) = \frac{m_w^1x + m_w^2}{m_w^3x+m_w^4}.
\]
The explicit formula of these parameters is useful in order to obtain numerical examples but it is not needed in the theoretical purpose, 
hence, we then omit it.
\end{rem}
\begin{proof}
First, assume that there exists a non-negative stationary equilibrium $(U,V)$. The resource concentration is given by $R = R_0-U-V$. Then,
for $W\in\{U,V\}$ non zero we have:
\begin{equation}\label{Usurvie}
A_w(R)W = 0.
\end{equation}
With this notation, \eqref{Usurvie} reads
\begin{equation}\label{explicitUsurvie}
\left\{
\begin{aligned}
 &\left(X_w^1(R^1)-\lambda^1\right)W^1 + \lambda^1 W^2 = 0\\
 &\lambda^2 W^1 + \left(X_w^2(R^2)-\lambda^2\right)W^2 = 0.
\end{aligned}
\right.
\end{equation}
Since $W^1\geq 0$ and $W^2\geq 0$, we obtain
$W^1>0$ and $W^2>0$ which yields:
\[
 \left(X_w^1(R_u^1)-\lambda^1\right)<0.
\]
Moreover, \eqref{explicitUsurvie} implies that $0$ is an eigenvalue of $A_w(R)$ implying that $\det\left( A_w(R) \right)=0$ which reads 
explicitly:

\begin{equation}\label{definitionF}
 \left(X_w^1(R^1)-\lambda^1\right)\left(X_w^2(R^2) -\lambda^2\right)=\lambda^1\lambda^2
\end{equation}
Finally, we define 
\[
 D_w=\{r>0,\; X_w^1(r)-\lambda^1<0\}
\]
and the function $F_w$ such that:
\[
 \left(X_w^1(R^1)-\lambda^1\right)\left(X_w^2(F_w(R^1)) -\lambda^2\right)=\lambda^1\lambda^2
\]
The function $X_w^2$ being injective, the function $F_w$ reads shortly :
\[
F_w= \left(X_w^2\right)^{-1} \circ g \circ X_w^1
\]
wherein we have set the function g as:
\[
 g(x)=\lambda^2+\frac{\lambda^1\lambda^2}{x-\lambda^1}.
\]
\end{proof}


At this step, we see that it is necessary that $R=(R^1,R^2)$ belongs to the graph $\mathcal{C}_w=\{(r,F_w(r)),\;r\in D_w\}$ for the 
species $w\in\{u,v\}$ to survive. But this is not a sufficient condition. Indeed, the definition of the functions $F_w$ 
correspond to the fact that $0$ is an {\it eigenvalue}\footnote{Indeed, on $D_w$ the eigenvalue $0$ is the principal eigenvalue of 
$A_w(R)$, and by the Perron-Frobenius theorem, it is associated to a positive eigenvector which is nothing but $U$.} of the matrix 
$A_w(R)$. 

The analysis of the corresponding eigenvector will give us sufficient conditions for a point of the curve to be a semi-trivial equilibrium 
(proposition \ref{prop:survival}) or a coexistence equilibrium (proposition \ref{prop:loca}). \\

For instance, assume that $(U,V)$ is a non-negative equilibrium of \eqref{eq:station2}. If $W\in\{U,V\}$) is non zero, then 
$R=(R^1,R^2)\in\mathcal{C}_w$ and $W$ is a positive eigenvectors of the matrix $A_w(R)$ for the eigenvalue $0$. It follows that there 
exists some scalar $\mu_w>0$ such that:
\begin{equation}\label{eq:eigenvector}
W= \mu_w \begin{pmatrix}
          \lambda^1\\
          -(X_w^1(R^1)-\lambda^1)
         \end{pmatrix}.
\end{equation}

\noindent In the case of the semi-trivial solution,  we have  $\begin{pmatrix}
                                                                R^1\\
                                                                R^2
                                                               \end{pmatrix}=R=R_0-W$ and it comes that:
\[
 R^2=R_0+\frac{1}{\lambda^1}(R_0-R^1)\left(X_w^1(R^1)-\lambda^1\right).
\]
This leadd us to define, for $w\in\{u,v\}$, the functions $g_w$ (defined on $D_w$) by:
\[ 
 g_w(r)=R_0+\frac{1}{\lambda^1}(R_0-r)\left(X_u^1(r)-\lambda^1\right).
\]

\begin{lem}\label{lem:gw} Let $w\in\{u,v\}$.
The function $g_w$ is increasing on the set $D_w$.
Moreover, if the semi-trivial stationary equilibrium $E_w$ exists  then the resource concentration  $R_w=\begin{pmatrix}
                                                                                                          R_w^1\\
                                                                                                          R_w^2
                                                                                                         \end{pmatrix}$ associated 
to $E_w$ satisfies $g_w(R_w^1) = R_w^2$.
\end{lem}
\begin{proof}The fact that $g_w(R_w^1) = R_w^2$ follows from the very definition of $g_w$.
A direct computation gives
\[
 g_w'(r) = -\frac{X_w^1(r)-\lambda^1}{\lambda^1} + (R_0-r)\frac{X_w^{1\prime}(r)}{\lambda^1}.
\]
Since $X_w^1(r)-\lambda^1<0$ for $r\in D_w$, it comes that $g_w$ is  increasing on $D_w$.
\end{proof}

We can now state the graphical characterization of the semi-trivial solution.
\begin{prop}\label{prop:survival}
Let $w\in\{u,v\}$. The semi-trivial solution $E_w$ exists if and only if there exists $R_w^1\in D_w$ such that $F_w(R_w^1)=g_w(R_w^1):=R_w^2$. In that case $E_w$ is unique and  the resource concentration at $E_w$ is $R_w=(R_w^1,R_w^2)$.
\end{prop}
\begin{proof}
The characterization of $R_w$ is a direct consequence of the proposition  \ref{prop:Fw} and the lemma \ref{lem:gw}. The uniqueness follows from the fact that $r\mapsto g_w-F_w$ is increasing on $D_w$.
\end{proof}

Now, let us study the case of the coexistence stationary equilibrium.
From the proposition \ref{prop:Fw}, if there exists a coexistence solution, that is a  positive solution $(U_c,V_c)$ to \eqref{eq:station2}, then there
exists $R_c^1\in D_u\cap D_v$ such that 
\[
 F_u(R^1_c) = F_v(R^1_c)=R^2_c.
\]
According to remark \ref{rem:F}, we obtain the following lemma.
\begin{lem}\label{prop:nbr} Suppose that $F_u\neq F_v$. Then there are at most two coexistence stationary equilibrium for the gradostat.
\end{lem}
There are at most two intersections between the curves of $F_1$ and $F_2$ but these intersections are not necessarily associated to a positive solution of \eqref{eq:station2}. Indeed, if $F_u(R^1) = F_v(R^1)$ then the coefficients of the 
eigenvectors are not necessarily of the same signs.
%

The following proposition gives a good location for an intersection between the curves of $F_u$ and $F_v$ to be associated with an 
admissible stationary equilibrium solution of \eqref{eq:station2}.

\begin{prop}\label{prop:loca}
Let $R_c$ be an intersection between the curves of $F_u$ and $F_v$. $R_c$ is associated to an admissible coexistence stationary equilibrium if and only if:
\[
 \left( R_u^1-R_v^1 \right)\left( R_u^2-R_v^2 \right) <0,
\]
and $R_c$ is in the rectangle $K$ defined as:
\[
K = [\min(R_u^1,R_v^2),\max(R_u^1,R_v^1)]\times[\min(R_u^2,R_v^1),\max(R_u^1,R_v^2)].
\]
\end{prop}

\begin{proof}
Let us define, for each semi-trivial equilibrium the following sets of $[0,R_0]^2$:
\[
 K_w = \{ (R^1,R^2)\in[0,R_0]^2,\quad \left( R_w^1-R^1 \right)\left( R_w^2-R^2 \right) <0 \}.
\]



We first prove that any intersection $R_c$ between the curves of $F_u$ and $F_v$ is in $K_u \cap K_v$. Recall that $R_w$ is the
associated resource concentration for the stationary equilibrium $E_w$.
According to \eqref{Usurvie}, $R_c$ is associated to a stationary coexistence equilibrium only if $\det(A_u(R_c))=0$ and $\det(A_v(R_c))=0$.
But we also know that $\det(A_u(R_u))=0$ and $\det(A_v(R_v))=0$ which finally implies that:
\begin{align*}
 &\left( X_u^1(R_c^1)-\lambda^1 \right)\left( X_u^2(R_c^2)-\lambda^2 \right) = \left( X_u^1(R_u^1)-\lambda^1 \right)\left( X_u^2(R_u^2)-\lambda^2 \right),\\
 &\left( X_v^1(R_c^1)-\lambda^1 \right)\left( X_v^2(R_c^2)-\lambda^2 \right) = \left( X_v^1(R_v^1)-\lambda^1 \right)\left( X_v^2(R_v^2)-\lambda^2 \right).
\end{align*}
The fact that the functions $X_w^j(R^j)-\lambda^j$ are increasing gives us that necessarily $R_c\in K_u \cap K_v$.

From the equation \eqref{eq:eigenvector} coupled to the fact that $R_c = R_0 -U_c -V_c$, it comes that the values of the 
concentration $(U_c,V_c)$ associated to $R_c$ are given by:
\begin{equation}\label{eq:coeff}
 U_c = \mu_u \begin{pmatrix}
          \lambda^1\\
          -(X_u^1(R_c^1)-\lambda^1)
         \end{pmatrix} \text{ and }
         V_c = \mu_v \begin{pmatrix}
          \lambda^1\\
          -(X_v^1(R_c^1)-\lambda^1)
         \end{pmatrix}
\end{equation}
where the coefficients $\mu_u$ and $\mu_v$ are given by:
\[
 \mu_w = \frac{1}{X_{\overline{w}}^1(R_c^1)-X_w^1(R_c^1)}\left( g_{\overline{w}}(R_c^1)-R_c^2 \right).
\]
We know that $X_w^1(R_c^1)-\lambda^1<0$ for each $i$. As a consequence, $(U_c,V_c)$ is an admissible coexistence stationary equilibrium
if and only if $\mu_u>0$ and $\mu_v>0$. Hence, if $R_c$ is associated to an admissible coexistence stationary equilbrium, we have:
\[
 \min\left( g_u(R_c^1), g_v(R_c^1) \right) \leq R_c^2 \leq \max\left( g_u(R_c^1), g_v(R_c^1) \right).
\]
Consequently, $R_c$ is associated to an admissible equilibrium if and only if,
\begin{equation}\label{eq:cond}
 R_c \in \Theta = K_u \cap K_v \cap \left\{ (R^1,R^2)\in[0,R_0]^2, \quad \min\left( g_u(R^1), g_v(R^1) \right) \leq R^2 \leq \max\left( g_u(R^1), g_v(R^1) \right) \right\}
\end{equation}
Recall that the functions $g_w$ are defined by:
\[
 g_w(R) = R_0 + (R_0-R)\frac{X_w^1(R)-\lambda^1}{\lambda^1}.
\]
We just saw that if $R_c$ is associated to an admissible coexistence stationary equilibrium, then $R_c\in \Theta$ 
(it is the condition \eqref{eq:cond}). Consequently, properties on the functions $g_w$ allows the following statements:

If $\left( R_u^1-R_v^1 \right)\left( R_u^2-R_v^2 \right) >0$, it can be checked that $\Theta = \varnothing$, implying that $R_c$
does not exist.


If $\left( R_u^1-R_v^1 \right)\left( R_u^2-R_v^2 \right) <0$, then $\Theta \subset K$ where $K$ is the rectangle defined by:
\[
K = [\min(R_u^1,R_v^2),\max(R_u^1,R_v^1)]\times[\min(R_u^2,R_v^1),\max(R_u^1,R_v^2)].
\]


\end{proof}

\begin{cor}\label{cor:signe}
Assume that $R_c$ is associated to an admissible coexistence stationary equilibrium. Then:
\[
 R_u^1 < R_v^1 \Leftrightarrow X_u^1(R_c^1) > X_v^1(R_c^1).
\]
\end{cor}
\begin{proof}
Assume that $R_u^1 < R_v^2$. Proposition \ref{prop:loca} implies that $R_u^2 > R_v^2$. The functions $g_w$ are increasing on
the set $[R_u^1 , R_v^2]$ and $g_u(R_1^1) > g_v(R_2^1)$ because $g_w(R_w^1) = R_w^2$. As a consequence,
\[
 g_v(R_c^1) < R_c^2 < g_u(R_c^1).
\]
In the proof of the proposition \ref{prop:loca}, we calculated the coexistence stationary equilibrium associated to $R_c$ and found 
out that $U_c$ and $V_c$ satisfy \eqref{eq:eigenvector} where
\[
 \mu_w = \frac{1}{X_{\overline{w}}^1(R_c^1)-X_w^1(R_c^1)}\left( g_{\overline{w}}(R_c^1)-R_c^2 \right).
\]
Since $U_c >0$ and $V_c>0$, we have 
$\mu_u>0$ and $\mu_v>0$  which yields $X_u^1(R_c^1) > X_v^1(R_c^1)$.
\end{proof}

To summarize, we can tell if an intersection $R_c$ between the curves of $F_u$ and $F_v$ is associated to an admissible coexistence
stationary equilibrium. Now, we state a criteria
for the existence of coexistence stationary equilibrium according to the stability of the semi-trivial equilibrium $E_u$ and $E_v$.

\begin{prop}\label{prop:stab}
The semi-trivial equilibrium $E_w$ is stable if and only if $F_{\overline{w}}(R_w^1)> R_w^2$.
\end{prop}
\begin{proof}
The stability of $E_w$ can be read on the Jacobian of $H$ evaluated in $E_w$. For sake of simplicity we give the proof for $E_u$.
A straightforward computation gives:
\[
DH(U,0) = \begin{pmatrix}
               A & B\\
               0 & C
              \end{pmatrix}
\]
where,
\[
 A = \begin{pmatrix}
             X_u^1(R_u^1)-\lambda^1 -U^1f_u^{1\prime}(R_u^1) & \lambda^1\\
             \lambda^2 & X_u^2(R_u^2)-\lambda^2 -U^2f_u^{2\prime}(R_u^2)
            \end{pmatrix}
\]
and
\[
 C = \begin{pmatrix}
             X_v^1(R_u^1)-\lambda^1 & \lambda^1\\
             \lambda^2 & X_v^2(R_u^2)-\lambda^2 
            \end{pmatrix}.
\]
Using the facts that
\[
\left(X_u^1(R_u^1)-\lambda^1\right)\left(X_u^2(R_u^2)-\lambda^2\right) = \lambda^1\lambda^2,
\]
and $X_u^i(R_u^{i})-\lambda^i <0$, a simple computation shows that the real part of the eigenvectors of $A$ are negative.
As a consequence, $E_u$ is stable if and only if the eigenvectors of $C$ have negative real part which gives the announced inequality 
(recall that $F_v = (X_v^2)^{-1}\circ g \circ X_v^1$ and $g(x) = \lambda^2+\frac{\lambda^1\lambda^2 x}{x-\lambda^1}$).
\end{proof}
\subsubsection{Proof of the theorem \ref{thm:eq2}}
\begin{proof}
Let us assume that $R_u^1 < R_v^1$.
The existence of coexistence stationary equilibrium is a simple consequence of proposition \ref{prop:stab} and the intermediate value theorem.
Let us prove it if $E_u$ and $E_v$ are both stable, then according to proposition \ref{prop:stab}, $F_{\overline{w}}(R_w^1)> R_w^2$ for each $i$.
Since $R_w^2 = F_w(R_w^1)$, it comes that:
\[
 F_u(R_v^1) - F_v(R_v^1)>0 \text{ and } F_v(R_u^1) - F_u(R_u^1)>0.
\]
Hence, the intermediate value theorem implies that $F_u$ and $F_v$ have an odd number of intersections.
According to proposition \ref{prop:nbr}, there are at most two intersections between the curves of $F_u$ and $F_v$. As a consequence there
exists a unique $R_c^1 \in [R_u^1,R_v^1]$ such that $F_u(R_c^1) = F_v(R_c^1)$. Since the functions $F_w$ are decreasing, one
can check that $R_u^2 > R_v^2$ and that $R_c^2 \in [R_v^2,R_u^2]$. Hence, proposition \ref{prop:loca} implies that $R_c$
is associated to an admissible coexistence stationary equilibrium. Figure \ref{fig:caspossibles} comes as an illustration for this statement.

%

The stability of the coexistence stationary equilibrium is  more difficult to obtain. The jacobian of 
$H$ evaluated in $(U_c,V_c)$ reads:
\[
 DH(U_c,V_c) = \begin{pmatrix}
                    X_u^1-\lambda^1-\beta_u^1 & \lambda^1 & -\beta_u^1 & 0\\
                    \lambda^2 & X_u^2-\lambda^2-\beta_u^2 & 0 & -\beta_u^2\\
                    -\beta_v^1 & 0 & X_v^1-\lambda^1-\beta_v^1 & \lambda^1\\
                    0 & -\beta_v^2 & \lambda^2 & X_v^2-\lambda^2-\beta_v^2
                   \end{pmatrix}
\]
where:
\[
 X_w^j = f_w^j(R_c^j)-\delta^j<0 \text{ and } \beta_w^j = U_w^{j,c}f_w^{j\prime}(R_c^j)>0.
\]

Note that $DH(U_c,V_c)$ is an irreducible matrix and it can be written:
\[
 DH(U_c,V_c) = \begin{pmatrix}
                    A & B\\
                    C & D
                   \end{pmatrix},
\]
where $A$ and $D$ are irreducible square matrices with positive off diagonal elements and $B$ and $C$ are diagonal matrix with negative diagonal elements.

Let $s(DH(U_c,V_c))$ be the maximum real part of the eigenvalues of $DH(U_c,V_c)$.
Following \cite{sw}, we now use a very strong following property dealing with these kind of matrices (which can be found in \cite{lesTANKS}):
Defined
\[
 \overline{DH(U_c,V_c)}= \begin{pmatrix}
                              A & -B\\
                              -C & D
                             \end{pmatrix}.
\]
Then $s(DH(U_c,V_c))<0$ if and only if $(-1)^k d_k>0$ for $k\in\{1,2,3,4\}$, where $d_i$ is the $i$-th principal minor of $\overline{DH(U_c,V_c)}$.

As a consequence, the signs of $d_1$, $d_2$, $d_3$ and $d_4$ characterize the stability of $DH(U_c,V_c)$. Firstly, we have $d_1 = X_u^1-\lambda^1-\beta_u^1<0$.
Next, we have:
 $$d_2 = \begin{vmatrix}
        X_u^1-\lambda^1-\beta_u^1 & \lambda^1\\
        \lambda^2 & X_u^2-\lambda^2-\beta_u^2
       \end{vmatrix}= -\beta_u^1(X_u^2-\lambda^2)-\beta_u^2(X_u^1-\lambda^1)+\beta_u^1\beta_u^2 >0. $$

An other straightforward computation gives:
\begin{align*}
 d_3 &= \begin{vmatrix}
        X_u^1-\lambda^1-\beta_u^1 & \lambda^1 & \beta_u^1\\
        \lambda^2 & X_u^2-\lambda^2-\beta_u^2 & 0\\
        \beta_v^1 & 0 & X_v^1-\lambda^1-\beta_v^1
       \end{vmatrix}\\
     &= -\beta_u^1 \beta_v^1 (X_u^2-\lambda^2-\beta_u^2)+(X_v^1-\beta_v^1-\lambda^1)d_2\\
     &= -\beta_u^1(X_u^2-\lambda^2)(X_v^1-\lambda^1)-\beta_u^2(X_u^1-\lambda^1)(X_v^1-\lambda^1)+\beta_u^1\beta_u^2(X_v^1-\lambda^1)+\beta_u^2\beta_v^1(X_u^1-\lambda^1)<0
\end{align*}

Obtaining the sign of $d_4$ requires heavy computations. A straight computation, similar to the one in \cite{sw} gives:
\begin{lem}
\[
 d_4 = \mu_u\mu_v\lambda^1\lambda^2f_u^{2\prime}f_v^{1\prime}\frac{X_u^1-\lambda^1}{X_v^1-\lambda^1}\left( X_u^2-X_v^1 \right)\left( \frac{F_u'(R_c^1)}{F_v'(R_c^1)}-1 \right).
\]
\end{lem}
\begin{proof}
A straightforward computation gives:
\begin{align*}
d_4 &= \begin{vmatrix}
       X_u^1-\lambda^1-\beta_u^1 & \lambda^1 & \beta_u^1 & 0\\
       \lambda^2 & X_u^2-\lambda^2-\beta_u^2 & 0 & \beta_u^2\\
       \beta_v^1 & 0 & X_v^1-\lambda^1-\beta_v^1 & \lambda^1\\
       0 & \beta_v^2 & \lambda^2 & X_v^2-\lambda^2-\beta_v^2
      \end{vmatrix}\\
    &= \beta_v^2 D_1-\lambda^2 D_2+(X_v^2-\lambda^2-\beta_v^2)d_3
\end{align*}
Where,
\begin{align*}
 D_1 &= \begin{vmatrix}
         X_u^1-\lambda^1-\beta_u^1 & \beta_u^1 & 0\\
         \lambda^2 & 0 & \beta_u^2\\
         \beta_v^1 & X_v^1-\lambda^1-\beta_v^1 & \lambda^1
        \end{vmatrix}\\
     &= -\beta_u^2(X_u^1-\lambda^1)(X_v^1-\lambda^1)+\beta_u^2\beta_u^1(X_v^1-\lambda^1)+\beta_u^2\beta_v^1(X_u^1-\lambda^1)-\beta1^1\lambda^1\lambda^2
\end{align*}
and,
\begin{align*}
 D_2 &= \begin{vmatrix}
         X_u^1-\lambda^1-\beta_u^1 & \lambda^1 & 0\\
         \lambda^2 & X_u^2-\lambda^2-\beta_u^2 & \beta_u^2\\
         \beta_v^1 & 0 & \lambda^1
        \end{vmatrix}\\
     &= -\lambda^1\beta_u^1(X_u^2-\lambda^2)-\lambda^1\beta_u^2(X_u^1-\lambda^1)+\lambda^1\beta_u^1\beta_u^2+\lambda^1\beta_u^2\beta_v^1.
\end{align*}
By making good use of the relation $(X_w^1-\lambda^1)(X_w^2-\lambda^2)=\lambda^1\lambda^2$, one can check that:
\[
 d_4 = \beta_u^1\beta_v^2\left[\left(X_u^2-\lambda^2\right)\left(X_v^1-\lambda^1\right)-\lambda^1\lambda^2\right]+\beta_u^2\beta_v^1\left[\left(X_u^1-\lambda^1\right)\left(X_v^2-\lambda^2\right)-\lambda^1\lambda^2\right].
\]

From $(X_w^1-\lambda^1)(X_w^2-\lambda^2)=\lambda^1\lambda^2$,
we infer
%

$$d_4=		\left(X_v^1-X_u^1\right)\left( \beta_u^1\beta_v^2\left(X_u^2-\lambda^2\right) - \beta_u^2\beta_v^1\left(X_v^2-\lambda^2\right) \right).$$
Recall that $\beta_w^j = U_w^{j,c}f_w^{j\prime}(R_c^j)$. According to proposition \ref{eq:coeff},
\[
 W_c = \mu_w \begin{pmatrix}
                \lambda^1\\
                -(X_w^1-\lambda^1)
               \end{pmatrix}.
\]
and the coefficients $\mu_w$ are positive. From this relation comes that:
\[
 \beta_u^1\beta_v^2 = -\mu_u\mu_v(X_v^1-\lambda^1)f_u^{1\prime}(R_c^1)f_v^{2\prime}(R_c^2) \text{ and } \beta_u^2\beta_v^1 = -\mu_u\mu_v(X_u^1-\lambda^1)f_u^{2\prime}(R_c^2)f_v^{1\prime}(R_c^1).
\]
For the sake of simplicity we will note $f_w^{j\prime}$ for $f_w^{j\prime}(R_c^j)$. It comes:
\[
 d_4 = \mu_u\mu_v(X_u^1-X_v^1)\left( f_u^{1\prime}f_v^{2\prime}\left(X_u^2-\lambda^2\right)\left(X_v^1-\lambda^1\right) - f_u^{2\prime}f_v^{1\prime}\left(X_u^1-\lambda^1\right)\left(X_v^2-\lambda^2\right) \right).
\]
Using once again the relation $(X_w^1-\lambda^1)(X_w^2-\lambda^2)=\lambda^1\lambda^2$ gives:
\[
 d_4 = \mu_u\mu_v\lambda^1\lambda^2\frac{X_v^1-\lambda^1}{X_u^1-\lambda^1}f_u^{2\prime}f_v^{1\prime}\left( X_u^2-X_v^1 \right)\left( \frac{f_u^{1\prime}f_v^{2\prime}}{f_u^{2\prime}f_v^{1\prime}} - \left( \frac{X_u^1-\lambda^1}{X_v^1-\lambda^1} \right)^2 \right).
\]
We are going to express the derivatives of the functions $f_w^j$ using the functions $F_w$. It starts from a realtion we already proved:
\[
 (X_w^1(R^1)-\lambda^1)(X_w^2(R^2)-\lambda^2)=\lambda^1\lambda^2 \Leftrightarrow R^2 = F_w(R^1).
\]
It comes that:
\[
 (X_w^1(R^1)-\lambda^1)(X_w^2(F(R^1))-\lambda^2)=\lambda^1\lambda^2.
\]
Derivating by $R^1$ gives:
\[
 \frac{f_w^{1\prime}(R^1)}{f_w^{j\prime}(F_w(R^1))} = -F_w'(R^1)\frac{X_w^1(R^1)-\lambda^1}{X_w^2(F_w(R^1))-\lambda^2}.
\]
Since $R_c^2 = F_1(R_c^1) = F_2(R_c^1)$ it comes that:
\[
 \frac{f_u^{1\prime}f_v^{2\prime}}{f_u^{2\prime}f_v^{1\prime}} = \frac{F_1'(R_c^1)}{F_2'(R_c^1)} \left( \frac{X_u^1-\lambda^1}{X_v^1-\lambda^1} \right)^2.
\]
Hence,
\[
 d_4 = \mu_u\mu_v\lambda^1\lambda^2f_u^{2\prime}f_v^{1\prime}\frac{X_u^1-\lambda^1}{X_v^1-\lambda^1}\left( X_u^2-X_v^1 \right)\left( \frac{F_u'(R_c^1)}{F_v'(R_c^1)}-1 \right).
\]
\end{proof}
As a direct consequence, the sign of $d_4$ is given by the sign of the quantity:
\[
 \text{sign}(d_4) = \left( X_u^2-X_v^1 \right)\left( \frac{F_u'(R_c^1)}{F_v'(R_c^1)}-1 \right).
\]
Moreover corollary \ref{cor:signe} gives us a better understanding of this sign:
\[
 \text{sign}(d_4) = \left( R_v^1-R_u^1 \right)\left( \frac{F_u'(R_c^1)}{F_v'(R_c^1)}-1 \right).
\]
Let us assume that $R_v^1-R_u^1>0$ (the proof is the same if we suppose that $R_v^1-R_u^1<0$). We will now show how the stability of the semi-trivial equilibrium $E_u$ and $E_v$ influence the 
stability of the coexistence stationary equilibrium when it exists.

If $E_u$ and $E_v$ are stable, then according to proposition \ref{prop:stab}, we have:
\[
 F_u(R_v^1) - F_v(R_v^1)>0 \text{ and } F_v(R_u^1) - F_u(R_u^1)>0.
\]
And we already know that there exists a unique intersection between the curves of $F_u$ and $F_v$ in the interval $[R_u^1,R_v^1]$.
A simple analytic consequence of these facts is that $F_v'(R_c^1) < F_u'(R_c^1)$ and since the functions $F_w$ are decreasing it comes that:
\[
 \frac{F_u'(R_c^1)}{F_v'(R_c^1)}-1 <0.
\]
Thus $d_4<0$ which implies that the unique coexistence equilibrium is unstable. 

This reasoning also proves the stability property of the coexistence stationary equilibrium in the other cases which concludes the proof.
\end{proof}

\bibliographystyle{abbrv}
\bibliography{Chemostat}
\end{document}